\documentclass[12pt, a4paper, leqno]{amsart}
\usepackage[utf8]{inputenc}
\usepackage{amsmath}
\usepackage{amsfonts}
\usepackage{amssymb}
\usepackage{times}
\usepackage[T1]{fontenc} %skandit 
\usepackage{url} 
\usepackage{color,esint}
\usepackage[dvipsnames]{xcolor}
\usepackage[colorlinks, allcolors=RedViolet,pdfstartview=,pdfpagemode=UseNone]{hyperref} 
\usepackage{graphicx} 
\usepackage{enumitem}
\usepackage{tabularx}
 \usepackage{mathtools}
\usepackage{pgf,tikz}
\usepackage{mathrsfs}
\usetikzlibrary{arrows}

\let\oldmarginpar\marginpar
\renewcommand\marginpar[1]{\-\oldmarginpar[\raggedleft\footnotesize #1]%
{\raggedright\footnotesize #1}}

\usepackage{amsthm}
\theoremstyle{plain}
\newtheorem{thm}[equation]{Theorem}
\newtheorem{lem}[equation]{Lemma}
\newtheorem{prop}[equation]{Proposition}
\newtheorem{cor}[equation]{Corollary}

\theoremstyle{definition}
\newtheorem{defn}[equation]{Definition}

\newtheorem{eg}[equation]{Example}

\theoremstyle{remark}
\newtheorem{rem}[equation]{Remark}

\numberwithin{equation}{section}

\newcommand{\R}{\mathbb{R}}

\newcommand{\Z}{\mathbb{Z}}
\newcommand{\Rn}{{\mathbb{R}^n}}

\def\supp{\operatornamewithlimits{supp}}

\def\dist{\qopname\relax o{dist}}

\DeclareMathOperator{\divop}{div}
\renewcommand{\div}{\divop}
\renewcommand{\phi}{\varphi}
\def\le{\leqslant}
\def\leq{\leqslant}
\def\ge{\geqslant}

\def\phi{\varphi}
\def\rho{\varrho}
\def\epsilon{\varepsilon}
\def\vartheta{\theta}

\def\px{{p(\cdot)}}

\def\loc{{\rm loc}}
\def\BV{{\rm BV}}

\newcommand{\inc}[1]{\hyperref[def:aInc]{{\normalfont(Inc){\ensuremath{_{#1}}}}}}
\newcommand{\dec}[1]{\hyperref[def:aDec]{{\normalfont(Dec){\ensuremath{_{#1}}}}}}
\newcommand{\ainc}[1]{\hyperref[def:aInc]{{\normalfont(aInc){\ensuremath{_{#1}}}}}}
\newcommand{\adec}[1]{\hyperref[def:aDec]{{\normalfont(aDec){\ensuremath{_{#1}}}}}}
\newcommand{\adeci}[1]{\hyperref[def:aDeci]{{\normalfont(aDec){\ensuremath{_{#1}^\infty}}}}}
\newcommand{\azero}{\hyperref[def:a0]{{\normalfont(A0)}}}
\newcommand{\aone}{\hyperref[def:a1]{{\normalfont(A1)}}}
\newcommand{\VA}{\hyperref[def:va1]{{\normalfont(VA1)}}}
\newcommand{\aones}[1]{\hyperref[def:a1s]{{\normalfont(A1-{\ensuremath{{#1}})}}}}

\newcommand{\Phiw}{\Phi_{\text{\rm w}}}
\newcommand{\Phic}{\Phi_{\text{\rm c}}}

\newcommand{\nablaa}{\nabla^a}

\date{\today}

\begin{document}

\title{Bounded variation spaces with generalized Orlicz growth related to image denoising}

\date{\today}

\author{Michela Eleuteri}
\address{Michela Eleuteri, Dipartimento di Scienze Fisiche, Informatiche e Matematiche, 
Università degli Studi di Modena e Reggio Emilia, Italy}
\email{\texttt{michela.eleuteri@unimore.it}}

\author{Petteri Harjulehto}
 % Address of record for the research reported here
 \address{Petteri Harjulehto,
 Department of Mathematics and Statistics,
FI-00014 University of Helsinki, Finland}
% Current address
\email{\texttt{petteri.harjulehto@helsinki.fi}}

\author{Peter Hästö}
 % Address of record for the research reported here
 \address{Peter Hästö, Department of Mathematics and Statistics,
FI-20014 University of Turku, Finland}
% Current address
\email{\texttt{peter.hasto@utu.fi}}

\begin{abstract}
Motivated by the image denoising problem and the undesirable stair-casing effect 
of the total variation method, we introduce bounded variation spaces with generalized Orlicz growth. 
Our setup covers earlier variable exponent and double phase models. 
We study the norm and modular of the new space and derive a formula for the modular in terms 
of the Lebesgue decomposition of the derivative measure and a location dependent recession function. 
We also show that the modular can be obtained as the $\Gamma$-limit of uniformly convex 
approximating energies. 
\end{abstract}


\maketitle

%%%%%%%%%%%%%%%%%%%%%%%%%%%%%%%%%%%%%%%%%%%%%%%%%%%%%%%%
%%%%%%%%%%%%%%%%%%%%%%%%%%%%%%%%%%%%%%%%%%%%%%%%%%%%%%%%
%%%%%%%%%%%%%%%%%%%%%%%%%%%%%%%%%%%%%%%%%%%%%%%%%%%%%%%%
\section{Introduction}

In PDE-based image processing, a function $u:\Omega\to \R$ represents 
the gray-scale intensity at each location of an image. Edges of objects 
correspond to discontinuities of $u$ and make this field challenging for function spaces 
and the calculus of variations. The space $\BV$ of functions of bounded variation has proven 
to be useful in the field. We refer to the book \cite{AubK06} by Aubert and Kornprobst for an overview.
The classical ROF image restoration/denoising model \cite{RudOF92} calls for minimizing the energy 
\[
\inf_{u\in\BV(\Omega)} \int_\Omega |Du| + |u-f|^2\, dx,
\]
where $f\in L^2(\Omega)$ is the given, corrupted input image that is to be restored. 
The \textit{fidelity term $|u-f|^2$} forces $u$ to be close to $f$ on average, whereas the 
\textit{regularizing term $|Du|$} limits the variation of $u$. This model suffers from 
a stair-casing effect that leads to piecewise constant minimizers \cite{ChaL97, Jal16}. For a recent 
overview of autonomous variants of the model we refer to \cite{PagPRV_pp}.

Image restoration has also been approached with non-autonomous energies that treat different locations 
differently. The first such model, by Chen, Levine and Rao \cite{CheLR06}, involves
the minimization of 
\begin{equation}\label{eq:CLR}
\min_{u\in \BV(\Omega)} \int_\Omega \phi_{clr}(x, |Du|)+ |u-f|^2\, dx,
\end{equation}
where the regularizing term has variable exponent growth for small energies and is given by 
\[
\phi_{clr}(x,t):= 
\begin{cases}
\tfrac1{p(x)}t^{p(x)}, &\text{when } t\in [0,1], \\ 
t - 1+\tfrac1{p(x)}, &\text{when } t>1. 
\end{cases} 
\]
The variable exponent $p:\Omega\to (1,2]$ is a function bounded away from $1$ 
(i.e.\ $p^-:=\inf p >1$) which should be chosen close to $2$ in smooth areas of the image and close 
to $1$ near likely edges to avoid stair-casing as well as blurring. 
Since $\phi(x,t)\sim t$ as $t\to\infty$, this model can be analyzed in the 
classical $\BV$-space. 
Furthermore, using the Lebesgue decomposition of the derivative measure $Du$, Chen, Levine and Rao define
\[
\int_\Omega \phi_{clr}(x, |Du|)\, dx
:=
\int_\Omega \phi_{clr}(x, |\nablaa u|)\, dx + |D^su|(\Omega),
\]
where $\nablaa u$ is the density of the absolutely continuous part of the derivative.
They prove for instance that
\begin{equation}\label{eq:CLRduality}
\int_\Omega \phi_{clr}(x, |Du|)\, dx
= 
\sup_{w\in C^1_0(\Omega; \Rn), |w|\le 1} \int_\Omega u \div w - \tfrac1{p'(x)} |w|^{p'(x)}\, dx 
\end{equation}
and use this duality formulation to prove existence and properties of minimizers of \eqref{eq:CLR}. 
The reason why we call this a duality formulation and the rationale behind the term $\tfrac1{p'(x)} |w|^{p'(x)}$ 
will become clear once we introduce a more general framework. 
%This is our first contribution. 

Subsequently, Li, Li and Pi \cite{LiLP10} proposed an image restoration model in the 
variable exponent space $W^{1,\px}(\Omega)$ with energy $\phi_\px(x,t):=t^{p(x)}$ and $p^->1$. 
The last restriction implies that the problem involves only reflexive Sobolev spaces and 
that the minimizers are $C^{1,\alpha}$, so theoretically it is ill-suited to the image 
processing context. 
Harjulehto, H\"ast\"o, Latvala and Toivanen \cite{HarHL08, HarHLT13} considered 
the same energy without the restriction $p^->1$. In this case, a relaxation procedure 
shows that the ``correct'' energy for $\BV$-functions is 
\[
\int_\Omega \phi_\px(x, |Du|)\, dx
:=
\int_\Omega \phi_\px(x, |\nablaa u|)\, dx + |D^su|(\{p=1\})
\]
provided $|D^su|(\{p>1\})=0$, analogously to the Chen--Levine--Rao formula \eqref{eq:CLRduality}. 

More recently, double phase energies have attracted the attention of many in the field 
of non-standard growth \cite{BaaBL22, BarCM18, ColM15a, DeF22, FarFW22, LiuP22, MizS21}. Most important for image processing 
is the version $\phi_{dp}(x,t):= t + a(x)t^2$ with $a\ge 0$ and powers $1$ and $2$. 
Harjulehto and H\"ast\"o \cite{HarH21} considered this energy with the interpretation 
\[
\int_\Omega \phi_{dp}(x, |Du|)\, dx
:=
\int_\Omega \phi_{dp}(x, |\nablaa u|)\, dx + |D^su|(\{a=0\})
\]
provided $|D^su|(\{a>0\})=0$. For instance they showed that it is the $\Gamma$-limit as $\epsilon\to 0^+$ of 
the uniformly convex approximating energies given by $\phi_\epsilon(x,t):=t^{1+\epsilon} + a(x)t^2$. 

The purpose of the present article is to introduce a general model which covers all these cases 
as well as countless variants like the perturbed variable exponent model and the Orlicz double phase model
(see \cite{HasO22a, HasO22b} for a list on variants with references). 
Generalized Orlicz spaces, also known as Musielak--Orlicz spaces, have been widely studied 
recently (see, e.g., \cite{ChlGSW21, HadSSV_pp, HurOS_pp, WanZ22, WeiXY22}). 
We consider a generalized $\Phi$-function $\phi:\Omega\times [0,\infty)\to[0,\infty]$ 
which may have linear growth at infinity at some points and superlinear growth at others. 
The dual space in the linear case is $L^\infty$ which can be seen in the restriction $|w|\le 1$ in 
\eqref{eq:CLRduality}. This space lacks several nice properties but it is nevertheless very concrete. 
However, to deal with the general case we consider the space $L^{\phi^*}(\Omega)$ given by the 
conjugate function $\phi^*$. Now the linearity of $\phi$ means that $\phi^*$ is not doubling; 
in fact, it is not even finite. Consequently, we can neither use the theory of 
doubling $\Phi$-functions, nor the concreteness of the space $L^\infty(\Omega)$. 
Fortunately, the theory of non-doubling variable exponent and generalized Orlicz spaces has 
been developed in \cite{DieHHR11, HarH19} 
and we know for instance that the maximal operator is bounded irrespective of doubling. 
Nevertheless, we need new types of approximation estimates that handle the transition 
between the $L^1$-, $L^p$- and $L^\infty$-regimes without extra constants which can ruin 
an argument in the non-doubling case. 
These techniques require subtly stronger assumptions on $\phi$, as the usual \aone{} 
does not suffice (see Example~\ref{eg:weightNeeded}).

Duality is a commonly used strategy in $\BV$-spaces and image restoration. We use it to define 
appropriate norms $V_\phi$ and modulars $\rho_{V,\phi}$ and study their properties in Section~\ref{sect:basic}. 
To our knowledge, this is the first time that the duality approach has been used to define a modular
in a Sobolev-type space. In Section~\ref{sect:approximation}, we consider 
approximation with respect to $V_\phi$ and the new space 
$\BV^\phi(\Omega)$ which generalizes $\BV(\Omega)$. The main result (Theorem~\ref{thm:exactFormula}) 
provides the formula 
\[
\rho_{V,\phi}(u) = \rho_\phi(|\nablaa u|) + \int_\Omega \phi'_\infty \, d|D^su|
\]
for the modular in terms of the \textit{recession function
$\phi'_\infty:\Omega\to[0,\infty]$} defined by 
\[
\phi'_\infty(x):= \limsup_{t\to \infty} \frac{\phi(x, t)} t.
\]
This function is often used in relaxation including in image processing (see, e.g., \cite{AmeGZ14,PagPRV_pp}). 
However, since we consider the non-autonomous case, our recession function 
depends on $x$ and so acts as a weight on the singular part of the function. 
For instance in the case $\phi(x,t):=t^{p(x)}$ we have $\phi'_\infty=1$ in the set $\{p=1\}$ and 
$\phi'_\infty=\infty$ elsewhere. This example shows that the continuity of $\phi$ does not ensure 
the continuity of $\phi'_\infty$. Furthermore, this makes the non-autonomous case much more 
difficult than the autonomous case, where the space $\BV^\phi$ reduces to classical $\BV$- or Sobolev spaces 
(see Corollary~\ref{cor:Orlicz}).

Using this formula we conclude the paper in Section~\ref{sect:Gamma} by showing the $\Gamma$-convergence 
of regularized functionals from \cite{EleHH_pp} to $\rho_{V,\phi}$. 
We start with background (Section~\ref{sect:background}) and auxiliary results (Section~\ref{sect:auxiliary}). 
A critical tool of independent interest is the Young convolution inequality 
with asymptotically sharp constants (Corollary~\ref{cor:convolution}).

%%%%%%%%%%%%%%%%%%%%%%%%%%%%%%%%%%%%%%%%%%%%%%%%%%%%%%%%
%%%%%%%%%%%%%%%%%%%%%%%%%%%%%%%%%%%%%%%%%%%%%%%%%%%%%%%%
%%%%%%%%%%%%%%%%%%%%%%%%%%%%%%%%%%%%%%%%%%%%%%%%%%%%%%%%

\section{Background}
\label{sect:background}

%%%%%%%%%%%%%%%%%%%%%%%%%%%%%%%%%%%%%%%%%%%%%%%%%%%%%%%%
\subsection*{Notation and terminology}

Throughout the paper we always consider a 
bounded domain $\Omega \subset \Rn$, i.e.\ an open and connected 
set. By $p':=\frac p {p-1}$ we denote the H\"older conjugate exponent 
of $p\in [1,\infty]$. The notation $f\lesssim g$ means that there exists a constant
$c>0$ such that $f\le c g$. The notation $f\approx g$ means that
$f\lesssim g\lesssim f$ whereas $f\simeq g$ means that 
$f(t/c)\le g(t)\le f(ct)$ for some constant $c \ge 1$. 
By $c$ we denote a generic constant whose
value may change between appearances.
A function $f$ is \textit{almost increasing} (more precisely, $L$-almost increasing) if there
exists $L \ge 1$ such that $f(s) \le L f(t)$ for all $s \le t$.
\textit{Almost decreasing} is defined analogously.
By \textit{increasing} we mean that the inequality holds for $L=1$ 
(some call this non-decreasing), similarly for \textit{decreasing}. 

Consider a function $\|\cdot\|: X \to [0, \infty]$ on a real vector space $X$ and the 
following conditions:
\begin{itemize}
\item[(N1)] $\|f\|=0$ implies that $f=0$.
\item[(N2)] $\|af\| = |a| \|f\|$ for all $f \in X$ and $a \in \R$;
\item[(N3)] $\|f+g\| \le \|f\|+ \|g\|$ for all $f, g \in X$.
\item[(N3$'$)] $\|f+g\| \lesssim \|f\|+ \|g\|$ for all $f, g \in X$.
\end{itemize}
We use the following terminology for $\|\cdot\|$:\\[4pt]
\centerline{
\begin{tabular}{rcccc}
&(N1)& (N2)& (N3) & (N3$'$) \\
\hline
\textit{quasi-seminorm} & &\checkmark&&\checkmark\\ 
\textit{seminorm} & &\checkmark&\checkmark&\\ 
\textit{quasinorm} & \checkmark&\checkmark&&\checkmark\\ 
\textit{norm} &\checkmark &\checkmark & \checkmark &\\ 
\end{tabular}}
%\centerline{
%\begin{tabular}{lcccc}
%&
%\textit{quasi-seminorm} & (N2), (N3$'$)\\ 
%\textit{seminorm} & (N2), (N3)\phantom{$'$}\\ 
%\textit{quasinorm} & (N1), (N2), (N3$'$)\\
%\textit{norm} & (N1), (N2), (N3)\phantom{$'$}\\
%\end{tabular}}

%%%%%%%%%%%%%%%%%%%%%%%%%%%%%%%%%%%%%%%%%%%%%%%%%%%%%%%%
\subsection*{Generalized Orlicz spaces}

We first define types of modulars that generate our spaces. 
Note that our terminology differs from Musielak \cite{Mus83}. 
Our justification is the following: a 
quasi-semimodular generates a quasi-seminorm, a semimodular generates a seminorm, etc.

\begin{defn}\label{def:quasiConvexsemimodular}
Let $X$ be a real vector space. A function $\rho:X \to [0,\infty]$ is called a \textit{quasi-semimodular} on $X$ if:
\begin{enumerate}
\item $\rho(0_{X})=0$;
\item the function $\lambda\mapsto\rho(\lambda x)$ is increasing on $[0,\infty)$ for every $x\in X$;
\item $\rho (-x)=\rho(x)$ for every $x\in X$;
\item there exists $\beta\in (0,1]$ such that $\rho(\beta(\alpha x +(1-\alpha)y) ) \leqslant \alpha\rho(x) + 
(1-\alpha)\rho(y)$ 
for every $x,y\in X$ and every $\alpha \in [0,1] $.
\end{enumerate}
If (4) holds with $\beta=1$, then $\rho$ is a \textit{semimodular}.
A (quasi-)semimodular is called a \textit{(quasi)modular} provided $\rho(x)=0$ if and only if $x=0_X$. 
% semimodular is called \textit{continuous } if the mapping $\lambda \mapsto \rho(\lambda x)$ is continuous on 
% $[0,\infty)$ for every fixed $x \in X$.
\end{defn}

\begin{defn}\label{def:modularSpace}
If $\rho$ is a quasi-semimodular in $X$, then the \textit{modular space
$X_\rho:=\{x \in X \mid \|x\|_\rho<\infty\}$} is defined by the quasi-seminorm
\[
\|x\|_{\rho}:= \inf \bigg\{\lambda >0 \,\Big|\,\rho\Big(\frac{x}{\lambda}\Big) \le 1 \bigg\}.
\]
\end{defn}

The next definitions are from \cite{HarH19}.
Our previous works were based on conditions defined for almost every point $x\in \Omega$. 
In this article we also use singular measures, so the assumptions are adjusted to hold for every point, 
following \cite{HasJR_pp}. We denote by $L^0(\Omega)$ the set of measurable functions in $\Omega$. 

\begin{defn}
\label{def2-1}
We say that $\phi: \Omega\times [0, \infty) \to [0, \infty]$ is a 
\textit{weak $\Phi$-function}, and write $\phi \in \Phiw(\Omega)$, if 
the following conditions hold for every $x \in \Omega$:
\begin{itemize}
\item 
$\phi(\cdot, |f|)$ is measurable for every $f\in L^0(\Omega)$.
\item
$t \mapsto \phi(x, t)$ is increasing. %non-decreasing.
\item 
$\displaystyle \phi(x, 0) = \lim_{t \to 0^+} \phi(x,t) =0$ and $\displaystyle \lim_{t \to \infty}\phi(x,t)=\infty$.
\item 
$t \mapsto \frac{\phi(x, t)}t$ is $L$-almost increasing on $(0,\infty)$ with 
constant $L$ independent of $x$.
\end{itemize}
If $\phi\in\Phiw(\Omega)$ is additionally convex and left-continuous with respect to $t$ for every $x\in\Omega$, then $\phi$ 
is a 
\textit{convex $\Phi$-function} and we write $\phi \in \Phic(\Omega)$. If $\phi$ does not depend on $x$, then we 
omit the set and write $\phi \in \Phiw$ or $\phi \in \Phic$.
\end{defn}

Since the range of $\phi$ is $[0,\infty]$, convexity can be defined as usual by the 
inequality 
\[
\phi(x, \theta t+(1-\theta)s)\le \theta\phi(x, t)+(1-\theta)\phi(x,s)
\] 
including the 
case $\infty\le\infty$. As we deal with conjugates of linear growth at infinity, it is crucial that 
we allow extended real-valued $\Phi$-functions. Chlebicka, Gwiazda and colleagues 
(e.g.\ \cite{BorC22, ChlGSW21}) have considered the case of non-doubling N-functions;
however, this is not sufficient here since N-functions exclude $L^1$- and $L^\infty$-spaces which are needed.

\begin{defn}
\label{rhophi}
Let $\varphi \in \Phiw(\Omega)$ and $\displaystyle
\rho_\phi(f) := \int_{\Omega} \varphi (x, |f|) \, dx$
for all $f \in L^0(\Omega)$. The set
\[
L^{\varphi}(\Omega) 
:= 
(L^0(\Omega))_{\rho_\phi}
=
\big\{f \in L^0(\Omega) \mid \rho_\phi(\lambda  f) < \infty \quad \textnormal{for some }\lambda > 0\big\}
\]
with quasinorm given by
$
\|f\|_\phi 
:=
\|f\|_{\rho_\phi}
%= \inf \left \{\lambda > 0 \,\Big|\, \rho_\phi \left (\frac{f}{\lambda} \right ) \le  1\right \}.
$
is called a \textit{generalized Orlicz space}.
We use the abbreviation $\|v\|_\phi := \big\| |v|\big\|_\phi$ 
for vector-valued functions.
\end{defn}

We observe that $\|\cdot\|_\phi$ is a quasinorm in $L^\phi(\Omega)$ 
if $\phi \in \Phiw(\Omega)$, and a norm if $\phi \in \Phic(\Omega)$ \cite[Lemma~3.2.2]{HarH19}. 
We define two Sobolev spaces; the space $L^{1,\phi}$ is sometimes denoted by $V^1L^\phi$, indicating the that 
first variation $\nabla u$ belongs to $L^\phi$. Note that $W^{1, \phi}(\Omega) = L^{1, \phi} (\Omega) \cap L^\phi(\Omega)$.

\begin{defn}
Let $\varphi \in \Phiw(\Omega)$. 
A function $u \in W^{1,1}(\Omega)$ belongs to 
the \textit{Sobolev space} $W^{1, \varphi}(\Omega)$ if $|u|, |\nabla u| \in L^\phi(\Omega)$
and to 
the \textit{Sobolev space} $L^{1, \varphi}(\Omega)$ if $|\nabla u| \in L^\phi(\Omega)$. 
The spaces are equipped with the (quasi)norms 
\[
\|u\|_{W^{1, \varphi}(\Omega)} := \|u\|_\phi + \|\nabla u\|_\phi
\quad\text{and}\quad
\|u\|_{L^{1, \phi}(\Omega)} := \|u\|_{L^1(\Omega)} + \|\nabla u\|_{L^{\phi}(\Omega)}.
\]
\end{defn}

When $\phi$ in a sub- or superscript is replaced by a real number (e.g., $L^{1,p}$ or $\rho_2$), 
this is an abbreviation for the $\Phi$-function $\phi(x,t)\equiv t^p$.

%%%%%%%%%%%%%%%%%%%%%%%%%%%%%%%%%%%%%%%%%%%%%%%%%%%%%%%%
%%%%%%%%%%%%%%%%%%%%%%%%%%%%%%%%%%%%%%%%%%%%%%%%%%%%%%%%
%%%%%%%%%%%%%%%%%%%%%%%%%%%%%%%%%%%%%%%%%%%%%%%%%%%%%%%%
\section{Auxiliary results}
\label{sect:auxiliary}

%%%%%%%%%%%%%%%%%%%%%%%%%%%%%%%%%%%%%%%%%%%%%%%%%%%%%%%%
\subsection*{Regularity conditions for harmonic analysis and PDE}

We say that $\omega: [0, \infty) \to [0, \infty]$ is a 
\textit{modulus on continuity} if it is increasing and
$\omega (0) = \lim_{t \to 0^+} \omega(t)=0$. Note that we do not require concavity and 
allow extended real values. 

For $\phi:\Omega\times [0,\infty)\to [0,\infty)$ and $p,q>0$ we define some conditions:
\begin{itemize}[leftmargin=4em]
\item[(A0)]\label{def:a0}
There exists $\beta \in(0, 1]$ such that $\phi(x, \beta) \le 1 \le \phi(x, \frac1\beta)$ 
for every \ $x \in \Omega$. 
%\item[(A1-$\omega$)] \label{def:a1s}
%if there exists $\beta \in (0,1]$ such that, for every ball $B$ and a.e.\ $x,y\in B \cap \Omega$,
%\[ 
%\phi(x,\beta t) \le \phi(y,t) \quad\text{when}\quad \omega_B^-(t) \in \bigg[1, \frac{1}{|B|}\bigg].
%\]
%\item[(A1-$s$)]
%if it satisfies \aones{\omega} for $\omega(x,t):=t^s$;
\item[(A1)]\label{def:a1}
%if it satisfies \aones{\phi};
For every $K>0$ there exists $\beta \in (0,1]$ such that, for every $x,y\in \Omega$,
\[ 
\phi(x,\beta t) \le \phi(y,t)+1 \quad\text{when}\quad \phi(y, t) \in \bigg[0, \frac{K}{|x-y|^n}\bigg].
\]
\item[(VA1)]\label{def:va1}
%if it satisfies \aones{\phi};
For every $K>0$ there exists a modulus of continuity $\omega$ such that, 
for every $x,y\in \Omega$,
\[ 
\phi(x,\tfrac t{1+\omega(|x-y|)}) \le \phi(y,t)+\omega(|x-y|) \quad\text{when}\quad \phi(y, t) \in \bigg[0, 
\frac{K}{|x-y|^n}\bigg].
\]
%
%\item[(A2)]\label{def:a2}
%if 
%there exists $\phi_\infty \in \Phiw$, $h \in L^1(\Omega)\cap L^\infty(\Omega)$, $s>0$ and $\beta \in (0, 1]$ such 
%that
%\[
%\phi(x, \beta t) \le \phi_\infty(t) + h(x) \quad \text{and} \quad \phi_\infty(t) \le \phi(x, \beta t) + h(x)
%\]
%for almost every $x \in \Omega$ when $\phi_\infty(t) \in [0, s]$ and $\phi(x, t) \in [0, s]$, respectively. 

\item[(aInc)$_p$] \label{def:aInc} 
There exists $L_p\ge 1$ such that 
$t \mapsto \frac{\phi(x,t)}{t^{p}}$ is $L_p$-almost 
increasing in $(0,\infty)$ for every $x\in\Omega$.

\item[(aDec)$_q$] \label{def:aDec}
There exists $L_q\ge 1$ such that 
$t \mapsto \frac{\phi(x,t)}{t^{q}}$ is $L_q$-almost 
decreasing in $(0,\infty)$ for every $x\in\Omega$.
\end{itemize} 
%
%Note that \atwo{} always holds if $\Omega$ is bounded.
%
%We say that $\phi$ satisfies \inc{p} if 
%\ainc{p} holds with $L_p=1$, similarly for \dec{q}.
We say that \ainc{} holds if \ainc{p} holds for some $p>1$, and similarly for \adec{}.
%
%If $\phi \in \Phiw(\Omega)$ satisfies \adec{}, then 
%\marginpar{Reference; or delete?}
%\begin{equation*}%\label{eq:almostId}
%\phi^{-1}(x, \phi(x, t))\approx \phi(x, \phi^{-1}(x, t)) \approx t.
%\end{equation*}
%The growth of the inverse is closely tied to that of the function: 
%$\phi$ satisfies \ainc{p} or \adec{q} 
%if and only if 
%$\phi^{-1}$ satisfies \adec{\frac 1p} or \ainc{\frac 1q} \cite[Section~2.3]{HarH19}.
% 
%For simplicity, we denote dependence on these conditions by $c(p,q,\ldots)$ with the understanding 
%that also $L_p$ and $L_q$ affect the constants. 

If $\phi\in \Phiw(\Omega)$, then 
$\phi(\cdot ,1)\approx 1$ implies \azero{}, and if $\phi$ satisfies \adec{}, then \azero{} and 
$\phi(\cdot ,1)\approx 
1$ are equivalent. For instance, $\phi(x, t)=t^p$ always satisfies 
\azero{}, since $\phi(x, 1) \equiv 1$. Assumption \aone{} is an almost continuity 
condition; in the variable exponent case $\phi(x, t):=t^{p(x)}$ it corresponds 
to $\log$-Hölder continuity of $\frac1p$ \cite[Proposition 7.1.2]{HarH19}. 
Finally, \ainc{} and \adec{} are quantitative versions of the $\nabla_2$ and $\Delta_2$ conditions and 
measure lower and upper growth rates. 

Note that the definition of \aone{} differs slightly from \cite{HarH19, Has15}, where it is assumed 
that 
\[ 
\phi(x,\beta t) \le \phi(y,t) \quad\text{when}\quad \phi(y, t) \in \bigg[1, \frac{1}{|B|}\bigg]
\]
and $x$ and $y$ belong to the ball $B$. 
If $\phi$ satisfies \azero{}, then this is equivalent to 
\[ 
\phi(x,\beta t) \le \phi(y,t)+1 \quad\text{when}\quad \phi(y, t) \in \bigg[0, \frac{1}{|B|}\bigg]
\]
and if $\phi$ satisfies \adec{}, then we can equivalently add in the $K$, as well, see \cite{Has_pp, HasO22b}. 

The ``vanishing \aone{}'' condition \VA{} is a continuity condition for $\phi$ which was introduced 
to prove maximal regularity of minimizers \cite{HasO22a}. In the variable exponent case it 
corresponds to vanishing $\log$-H\"older continuity. 
We need the following weaker version of \VA{} where at least one of the points 
has to belong to the set $\{\phi'_\infty<\infty\}$ defined using the recession function:

\begin{defn}
We say that $\phi \in \Phiw(\Omega)$ satisfies \textit{restricted \VA{}} if 
it satisfies \aone{} and 
for every $K>0$ there exists a modulus of continuity $\omega$ such that
\begin{equation*}
\phi(x,\tfrac t{1+\omega(|x-y|)}) \le \phi(y,t)+\omega(|x-y|) \quad\text{when}\quad \phi(y, t) \in \bigg[0, 
\frac{K}{|x-y|^n}\bigg]
\end{equation*}
for 
every $x,y\in \Omega$ with $\phi'_\infty(x) <\infty$ or $\phi'_\infty(y) <\infty$.
\end{defn} 

In \cite[Section~3]{HarHL08}, it was shown that $\log$-H\"older continuity in the variable exponent 
case was not sufficient for $\BV$-type spaces and a strong 
$\log$-H\"older continuity condition was introduced. 
As mentioned above, $\varphi(x,t) = t^{p(x)}$ satisfies 
\aone{} if and only if $\frac1p$ is $\log$-H\"older continuous. We now 
prove a corresponding connection between restricted \VA{} and strong $\log$-H\"older continuity.
For simplicity, only the case of finite exponents is considered. 

\begin{prop}\label{prop:strong}
Let $\phi(x,t):=t^{p(x)}$ be a variable exponent energy with $p:\Omega\to[1,\infty)$.
Then restricted \VA{} is equivalent to the \emph{strong $\log$-H\"older continuity} of $\frac1p$, i.e.\ 
$\log$-H\"older continuity with
\[
\lim_{x\to y} \big|1-\tfrac1{p(x)}\big| \log \tfrac{1}{|x-y|} = 0
\]
uniformly in $y\in \{p=1\}$.
\end{prop}

\begin{proof}
The connection between $\log$-H\"older continuity and \aone{} was 
established in \cite[Proposition 7.1.2]{HarH19}, so it only remains to consider the vanishing 
$\log$-H\"older continuity around the set $\{p=1\}$. 
Suppose that $p(y)=1$ or, equivalently, $\phi'_\infty(y)<\infty$. Then $\phi(y,t)=t^1=t$. 
First we assume restricted \VA{} with $K=1$  and modulus of continuity $\omega$.
Choosing $t := |x - y|^{-n}\ge 1$ and denoting $r:=|x-y|$, we have
\[ 
\big(\tfrac{t}{1 + \omega(r)}\big)^{p(x)} 
=
\phi(x,\tfrac t{1+\omega(r)}) \le \phi(y,t)+\omega(r) 
\le (1+\omega(r)) t.
\]
Taking the logarithm of the equivalent inequality $t^{p(x)-1}\le (1+\omega(r))^{p(x)+1}$, we find that 
\[
\big|1-\tfrac1{p(x)}\big| \log \tfrac{1}{|x-y|} 
\le \tfrac{1+ p(x)}{np(x)} \log(1 + \omega(r)) 
\le \tfrac{2}{n} \log(1 + \omega(r)) \to 0
\]
as $r\to 0^+$. Thus $p$ is strongly $\log$-H\"older continuous.

Assume conversely that $p$ is strongly $\log$-H\"older continuous so that 
\[
\omega_p(r):=
\sup_{y\in\{p=1\},\, x\in B_r(y)} \big|1-\tfrac1{p(x)}\big| \log \tfrac{1}{|x-y|}
\to 0
\]
as $r\to 0^+$. To establish the restricted \VA{}-condition when $p(y) = 1$, 
it is enough that
\[
t^{p(x) - 1} \le  (1 + \omega(r))^{p(x)}.
\]
The inequality is trivial if  $t \in [0, 1]$.  So let $t>1$.
The left-hand side is increasing in $t$ so the worst case is when $t = \frac{K}{|x-y|^n}$
and we can choose $\omega$ based on the estimate
\[
\begin{split}
t^{1-\frac1{p(x)}}-1 
&
%\le 
%\Big(\frac{K}{r^n}\Big)^{1-\frac1{p(x)}}-1 
\le 
\Big(\frac{K}{r^n}\Big)^{\frac{\omega_p(r)}{\log \frac1r}}-1 
=
e^{\frac{\log K + n\log \frac1r}{\log \frac1r} \omega_p(r)}-1 
\le
e^{(\log K + n) \omega_p(r)} - 1
=: \omega(r)
\end{split}
\]
when $r\le \frac1e$. 
The strong $\log$-H\"older continuity ensures that this tends to zero when $x\to y$. 
On the other hand, if $p(x) = 1$ in the \VA{}-condition, then we need 
\[
\frac{t}{1 + \omega(r)} \le t \le t^{p(y)} + \omega(r),
\]
which holds since $\sup_{t\ge 0} (t - t^{p(y)}) = p(y)^{-p'(y)}(p(y)-1) \le 1-\frac1{p(y)} 
\le \omega_p(r) \le \omega(r)$ for all small $r>0$.
\end{proof}

%We recall two versions of Young's convolution inequality. Let $\eta$ be the standard mollifier and 
%$\epsilon>0$. 
%If $\phi\in \Phiw(\Omega)$ satisfies \azero{} and \aone{}, then 
%\[
%\| f*\eta_\epsilon\|_\phi \lesssim \|f\|_\phi
%\]
%for all $f\in L^\phi(\Omega)$ \cite[Lemma~4.4.6]{HarH19}. 
%If $\phi\in \Phic(\Omega)$ satisfies \azero{} and \VA{}, then we have the following modular version
%\begin{equation}\label{eq:youngsConvolutionInequalityVA}
%\phi^-_{B(x,\epsilon)}\big(\tfrac1{1+\omega(\epsilon)} f*\eta_\epsilon\big) 
%\le 
%(\phi(\cdot, f)*\eta_\epsilon)(x) + \omega(\epsilon)
%\end{equation}
%by \cite[Theorem~2.3]{HasJR_pp}. The theorem in \cite{HasJR_pp} is proved for the case 
%$\eta=\frac{\chi_B}{|B|}$, but the same proof works for the standard mollifier as well.

%%%%%%%%%%%%%%%%%%%%%%%%%%%%%%%%%%%%%%%%%%%%%%%%%%%%%%%%%%%%%%%%%%%%%%%%%
\subsection*{Inequalities with sharp constants}

Analogues of Jensen's inequality \cite[Theorem~4.3.2]{HarH19} and 
Young's convolution inequality \cite[Lemma~4.4.6]{HarH19} are known in the generalized 
Orlicz space under the \aone{} assumption, but only with constants $\beta\ll 1$. 
Here we show that the \VA{} assumption lets us choose the constant $\beta\to 1^-$ at the 
price of restricting to a small ball. 
The next result is an improvement of \cite[Theorem~2.3]{HasJR_pp}. Note that we do not 
assume \adec{}. This makes the proof more difficult but is critical to the application 
in this article. 

\begin{thm}[Jensen's inequality]\label{thm:jensen}
If $\phi\in \Phic(\Omega)$ satisfies \VA{} and $\mu$ is a probability measure in the ball $B=B_r$
with $|B|\, \|\frac{d\mu}{dx}\|_\infty=:m<\infty$, then 
\[
%\phi_B^+\bigg(\frac1{(1+\omega(r))^2} \int_{B\cap \Omega} |f|\, d\mu\bigg)
%\le
\phi_B^-\bigg(\frac1{1+\omega(r)} \int_{B\cap \Omega} |f|\, d\mu\bigg)
\le
\int_{B\cap \Omega} \phi(x, f)\, d\mu + \omega(r),
\]
where $\omega$ be the modulus of continuity from \VA{} with $K:=m\rho_\phi(f)+2$
and $r>0$ is so small that $\omega(r)\le \frac1{|B|}$.
\end{thm}
\begin{proof}
We define $t_0:= \int_{B\cap \Omega}|f|\, d\mu$. By \cite[Lemma~4.3.1]{HarH19}, there exists $\beta>0$ such that 
\[
\phi_B^-\bigg(\beta\int_{B\cap \Omega} |f|\, d\mu\bigg)
\le
\int_{B\cap \Omega} \phi(x, f)\, d\mu.
\]
If $t_0=\infty$, this implies that the right-hand side of the claim is 
infinite so there is nothing to prove. Thus we may assume that $t_0<\infty$. 

Denote by $\phi'$ the left-continuous function, increasing in $s$, with 
\[
\phi(x,t) = \int_0^t \phi'(x,s)\, ds.
\]
Such a function exists since $\phi$ is convex in the second variable.
Fix $x_0\in B$ with 
\[
\tfrac1{1+\omega(r)} \phi'\big(x_0, \tfrac1{1+\omega(r)} t_0\big)\le (\phi')_B^-\big(  \tfrac1{1+\omega(r)} t_0\big)
\]
and assume $\beta\le \frac1{1+\omega(r)}$ is so small that $\phi(x_0,\beta t_0) \le \frac K{|B|}$.

We define $\psi\in\Phic$ by
\[
\psi(t):= \int_0^t \phi'(x_0, \min\{s, \beta t_0\})\, ds\,; 
\]
$\psi$ is convex since $\psi'$ is increasing. Furthermore, $\psi(\beta t)=\phi(x_0, \beta t)$ if $t\le t_0$. 
When $t\le t_0$ we consider two cases to show that 
\[
\psi(\beta t) 
%=\phi(x_0, \beta t) 
\le \phi(x, t)+\omega(r)
\]
for $x\in B$: if $\phi(x, t)\le \frac K{|B|}$ this follows from \VA{} and otherwise it follows 
from $\phi(x_0, \beta t) \le \frac K{|B|}\le \phi(x, t)$. When 
$t> t_0$ we estimate
\begin{align*}
\psi(\beta t) 
&= 
\psi(\beta t_0) + \beta(t-t_0)\phi'(x_0, \beta t_0)
\le
\phi(x, t_0)+\omega(r) + (t-t_0)(\phi')_B^-(t_0) \\
&\le
\phi(x, t_0)+\omega(r) + (t-t_0) \phi'(x, t_0)
\le
\phi(x, t) + \omega(r),
\end{align*}
where we also used the convexity of $\phi$ in the last step. 

It follows from Jensen's inequality for $\psi$ that 
\[
\phi\big(x_0, \beta t_0\big)
=
\psi\bigg(\beta\int_{B\cap \Omega} |f|\, d\mu\bigg) 
\le 
\int_{B\cap \Omega} \psi(\beta|f|)\, d\mu
\le 
\int_{B\cap \Omega} \phi(x, |f|)\, d\mu + \omega(r). 
\]
This is the claim once we show that we can choose $\beta=\frac1{1+\omega(r)}$. 
Since the integral on the right-hand side can be estimated by $\frac m{|B|}\rho_\phi(f)=\frac{K-2}{|B|}$ and 
$\omega(r)\le \frac1{|B|}$, the inequality gives $\phi(x_0, \beta t_0) \le \frac{K-1}{|B|}$. 
To summarize, we have shown that $\phi(x_0, \beta t_0) \le \frac{K}{|B|}$ implies 
$\phi(x_0, \beta t_0) \le \frac{K-1}{|B|}$. 

We next investigate how large we can make $\beta$. 
Consider the set 
\[
\Theta := \Big\{\theta \in (0,1] \,\Big|\, \phi\big(x_0, \tfrac\theta{1+\omega(r)} t_0\big) \le \tfrac K{|B|}\Big\}.
\]
Since $\phi(x_0, t)\to 0$ when $t\to 0^+$, the set is non-empty. 
If $\theta_k\in\Theta$ with $\theta_k\nearrow \theta_0$, then the left-continuity of $\phi(x_0,\cdot)$ implies
that $\theta_0\in \Theta$. 
If $\sup \Theta = 1$, then this means that the previous Jensen inequality holds for 
$\beta=\frac1{1+\omega(r)}$ and the claim is proved. 
Suppose then that $\theta_0:=\sup \Theta \in (0,1)$. 
For $\theta_0<\theta$, this implies that 
\[
\phi\big(x_0, \tfrac{\theta_0 }{1+\omega(r)}t_0\big) \le 
\tfrac{K-1}{|B|} < 
\tfrac K{|B|} < 
\phi\big(x_0, \tfrac\theta{1+\omega(r)}t_0\big).
\]
Since $\phi(x_0,\cdot)$ is convex, such discontinuity is 
only possible if the right-hand side equals infinity for every $\theta>\theta_0$.
If $\phi(x_0, \frac1{1+\omega(r)} t_0)=\infty$, then 
\[
\infty = \phi'(x_0, \tfrac1{1+\omega(r)} t_0) \le (1+\omega(r)) (\phi')_B^-(\tfrac1{1+\omega(r)} t_0)
\] 
by the choice of $x_0$. 
It follows that $\phi(x, t_0)=\infty$ for every $x\in B$. 
The set $A:=\{x\in B\mid |f(x)|\ge t_0\}$ 
has positive $\mu$-measure since $t_0$ is the $\mu$-average of $|f|$. Thus also 
$\int_{B\cap \Omega} \phi(x, |f|)\, d\mu =\infty$, so the claim holds in the form $\infty\le\infty$ in 
this case.
\end{proof}

The convolution in the next result should be understood as 
\[
f*\eta(x) := \int_\Omega f(y)\eta(x-y)\, dy
\]
to account for the fact that $f$ and $\phi$ are only defined in $\Omega$. Extending $\phi$ 
outside $\Omega$ while preserving \VA{} is non-trivial, but luckily that is not needed here.

\begin{cor}[Young's convolution inequality]\label{cor:convolution}
Let $\phi\in \Phic(\Omega)$ satisfy \VA{} and $\eta$ be the standard mollifier. 
Then there exists a modulus of continuity $\omega$ such that 
\[
\rho_\phi\big(\tfrac1{1+\omega(\delta)} f*\eta_\delta\big)
\le
\rho_\phi(f) + \omega(\delta)
\]
for every $\delta>0$.
\end{cor}
\begin{proof}
We may assume that $\rho_\phi(f)<\infty$ since otherwise there is nothing to prove. 
Let $\omega$ be the modulus of continuity from \VA{} with $K:=m\rho_\phi(f)+2$
and let  $r>0$ be so small that $\omega(r)\le \frac1{|B_r|}$.  Thus 
Theorem~\ref{thm:jensen} yields
\[
\phi_{B_r}^-\bigg(\frac{|f*\eta_r(x)|}{1+\omega(r)} \bigg) 
\le
(\phi(\cdot, f)*\eta_r)(x) + \omega(r).
\]
This yields 
\[
\phi_{B_r}^-\bigg(\frac{|f*\eta_r(x)|}{1+\omega(r)} \bigg) 
\le \frac m{|B_r|} \int_{B_r\cap \Omega} \phi(x, f) \, dx  + \frac1{|B_r|} < \frac{K}{|B_r|}.
\]
Thus  we obtain by \VA{} that
\[
\phi\bigg(x, \frac{|f*\eta_r(x)|}{(1+\omega(r))^2} \bigg)
\le
\phi_{B_r}^-\bigg(\frac{|f*\eta_r(x)|}{1+\omega(r)} \bigg)  + \omega(r)
\le
(\phi(\cdot, f)*\eta_r)(x) + 2\omega(r).
\]
%
%The restriction $\omega(r)\le \frac1{|B_r|}$ yields an upper bound $\delta_0$ for $r$. 
We integrate this over $\Omega$ and use Fubini's Theorem to conclude that
\[
\rho_\phi\big(\tfrac1{(1+\omega(r))^2} f*\eta_r\big)
\le
\int_\Omega \phi(x, f)*\eta_r \, dx + 2|\Omega| \omega(r)
\le 
\int_\Omega \phi(x, f) \, dx + 2|\Omega|\, \omega(r).
\]
This gives the claim with the modulus of continuity 
$\hat\omega(r):= \max\{2\omega(r)+\omega(r)^2, 2|\Omega|\, \omega(r)\}$; when 
$\omega(r)>\frac1{|B_r|}$ we set $\hat\omega(r):=\infty$.
\end{proof}

%%%%%%%%%%%%%%%%%%%%%%%%%%%%%%%%%%%%%%%%%%%%%%%%%%%%%%%%
\subsection*{Associate spaces and conjugate modulars}

The associate space is a variant of the dual function space which works better at the end-points 
$p=1$ and $p=\infty$. 
We define the \textit{associate space $(L^\phi(\Omega))'\subset L^0(\Omega)$} by the norm
\[
\|u\|_{(L^\phi(\Omega))'}
:=
\sup_{\|v\|_\phi \le 1} \int_\Omega uv\, dx.
\]
According to \cite[Theorem~3.4.6]{HarH19}, $(L^\phi(\Omega))'= L^{\phi^*}(\Omega)$ for  
$\phi\in \Phiw(\Omega)$, where
\[
\phi^*(x, t) 
:=
\sup_{s\ge 0} (st - \phi(x, s)). 
\]
The conjugate function $\phi^*$ has the following properties:
\begin{itemize}
\item If $\phi \in \Phiw(\Omega)$, then $\phi^*\in \Phic(\Omega)$, so $\phi^*$ is always convex 
and left-continuous \cite[Lemma~2.4.1]{HarH19}.
\item For $p,q\in (1,\infty)$, $\phi$ satisfies \ainc{p} or \adec{q} if and only if $\phi^*$ 
satisfies \adec{p'} or \ainc{q'}, respectively \cite[Proposition~2.4.9]{HarH19}.
\item If $\phi \in \Phic(\Omega)$, then $\phi^*(x, \frac{\phi(x, t)}{t}) \le \phi(x, t)$ \cite[p.~35]{HarH19}
and $\phi^{**}=\phi$ \cite[Corollary~2.6.3]{DieHHR11}. 
\item If $\phi$ satisfies \azero{} or \aone{}, then so does $\phi^*$ \cite[Lemmas~3.7.6 and 4.1.7]{HarH19}. 
\end{itemize}
It is well-known that ``Young's equality''
\[
t\phi'(t) = \phi(t) + \phi^*(\phi'(t))
\]
holds when $\phi\in\Phic$ is continuously differentiable. In fact, we can prove it for any 
sub-gradient even without assuming convexity:

\begin{lem}\label{lem:dual-equality}
Let $\phi\in\Phiw$. 
If $\phi(s)\ge \phi(s_0)+k(s-s_0)$ for all $s\ge 0$, then $\phi^*(k)=ks_0-\phi(s_0)$.
\end{lem}
\begin{proof}
We observe that 
\[
ks_0-\phi(s_0)
\le
\sup_{s\ge 0}(sk-\phi(s))
\le 
\sup_{s\ge 0}(sk-(\phi(s_0)+k(s-s_0)))
=
ks_0-\phi(s_0).
\]
Therefore, $\phi^*(k)=\sup_{s\ge 0}(sk-\phi(s))=ks_0-\phi(s_0)$. 
\end{proof}

Every $\phi\in\Phic$ can be represented as 
\[
\phi(t) = \int_0^t \phi'(\tau)\, d\tau. 
\]
The function $\phi':\Omega\times[0,\infty)\to [0,\infty)$ can be the left-continuous 
left-derivative, the right-continuous right-derivative, or something in between. 
In any case, $\phi(s)\ge \phi(s_0)+\phi'(s_0)(s-s_0)$ and so the previous lemma 
implies that 
\[
\phi^*(\phi'(t))= t \phi'(t)-\phi(t).
\]
Additionally, it is known that $\phi(t)\approx t \phi'(t)$ if $\phi$ satisfies \adec{} \cite[Lemma 3.3]{HarHJ23}.

%%%%%%%%%%%%%%%%%%%%%%%%%%%%%%%%%%%%%%%%%%%%%%%%%%%%%%%%
\subsection*{Functions of bounded variation}

A function $u \in L^1(\Omega)$ has \textit{bounded variation}, denoted $u \in \BV(\Omega)$, if
\[
V(u, \Omega) := \sup \bigg\{ \int_\Omega u \div w \, dx \, \Big|\, w\in C^1_0(\Omega;\Rn), 
|w|\le 1 \bigg\} < \infty.
\]
Such functions have weak first derivatives which are Radon measures which we denote $Du$. 
By \cite[Proposition~3.6]{AmbFP00}, $V(u,\Omega)$ equals the total variation $|Du|(\Omega)$
of the measure $Du$, defined as 
\[
|Du|(A):=\sup_{\cup A_i = A} \sum_i |Du(A_i)|
\] 
where the supremum is taken over finite partitions of $A$ by measurable sets $A_i$. 
Furthermore, we use the Lebesgue decomposition 
\begin{equation*}%\label{eq:decomposition}
Du = D^a u + D^s u,
\end{equation*}
where $D^a u$ is the absolutely continuous part of the derivative and $D^s u$ is the singular part. 
The density of $D^a u$ is the vector valued function $\nablaa u$ such that 
\[
\int_\Omega w \cdot d D^a u = \int_\Omega w \cdot \nablaa u \, dx
\]
for all $w \in C^\infty_0(\Omega; \Rn)$.
The space $\BV$ has the following compactness-type property  \cite[Proposition~3.13]{AmbFP00}:
if $\sup_i \big(\|u_i\|_{L^1(\Omega)} + |Du_i|(\Omega)\big)<\infty$, then 
there exists a subsequence and $u\in \BV(\Omega)$ such that 
\begin{equation*}%\label{eq:precomactness}
u_{i_j}\to u\text{ in } L^1(\Omega)
\quad\text{and}\quad
|Du|(\Omega)\le \liminf_{j \to \infty} |Du_{i_j}|(\Omega).
\end{equation*}
We refer to \cite{AmbFP00} for more information about $\BV$ spaces. 
The next lemma shows that the equality $V(u,\Omega)=|Du|(\Omega)$ holds separately 
for the singular part. 

\begin{lem}\label{lem:singularPart}
If $u\in \BV(\Omega)$, then 
$\displaystyle|D^su|(\Omega)=\sup\bigg\{ \int_\Omega w\cdot dD^su \,\Big|\, w \in C^1_0(\Omega; \Rn), |w|\le 1\bigg\}$.
\end{lem}
\begin{proof}
By $|Du|(\Omega)=V(u,\Omega)$, the definition of the weak derivative and $Du=D^au+D^su$, we see that
\begin{align*}
|Du|(\Omega)
&=
\sup_{w \in C^1_0(\Omega; \Rn), |w|\le 1}\bigg( \int_\Omega w\cdot dD^au + \int_\Omega w\cdot dD^su \bigg) \\
%& \le 
%\sup_{w \in C^1_0(\Omega; \Rn), |w|\le 1}\int_\Omega w\cdot dD^au + \sup_{w \in C^1_0(\Omega; \Rn), |w|\le 
%1}\int_\Omega w\cdot dD^su \\
&\le 
\sup_{w \in L^\infty(\Omega; \Rn), |w|\le 1}\int_\Omega w\cdot dD^au + \sup_{w \in L^\infty(\Omega; \Rn), |w|\le 1}\int_\Omega w\cdot dD^su \\
&= 
|D^au|(\Omega) + |D^su|(\Omega). 
\end{align*}
Since $|D^au|(\Omega) + |D^su|(\Omega) = |Du|(\Omega)$ as $D^a$ and $D^s$ are mutually singular, 
each inequality has to be an equality, and so the claim follows. 
\end{proof}

%%%%%%%%%%%%%%%%%%%%%%%%%%%%%%%%%%%%%%%%%%%%%%%%%%%%%%%%
%%%%%%%%%%%%%%%%%%%%%%%%%%%%%%%%%%%%%%%%%%%%%%%%%%%%%%%%
%%%%%%%%%%%%%%%%%%%%%%%%%%%%%%%%%%%%%%%%%%%%%%%%%%%%%%%%
\section{Basic properties of dual norms and modulars}
\label{sect:basic}

In this section we use a duality approach to define a norm and a modular. The 
``dual norm'' $V_\phi$ is related to the associate space and H\"older's inequality, whereas 
the ``dual modular'' $\rho_{V,\phi}$ is related to Young's inequality. 
Note that $V_\phi$ is not the norm generated by $\rho_{V,\phi}$; their 
relationship is explored in Lemma~\ref{lem:equivalence}.

\begin{defn}\label{defn:V-rho}
Let $\phi \in \Phiw(\Omega)$. For $u\in L^1(\Omega)$, we define the ``dual norm''
\[
V_\phi(u, \Omega):=
V_\phi(u):=\sup\bigg\{ \int_\Omega u \div w \, dx \,\Big|\, w \in C^1_0(\Omega; \Rn), \| w \|_{\phi^*}\le 1 \bigg\}
\]
and the ``dual modular''
\[
\rho_{V,\phi}(u):=\sup\bigg\{ \int_\Omega u \div w - \phi^*(x, |w|)\, dx \,\Big|\, w \in C^1_0(\Omega; \Rn)\bigg\}.
\]
We say that $u\in L^\phi(\Omega)$ belongs to $\BV^\phi(\Omega)$ if
\[
 \|u\|_{\BV^\phi}:= \|u\|_\phi + V_\phi(u) < \infty.
\]
\end{defn}

The next example shows that this definition is an extension of the ordinary $\BV$-space, in which the norm and 
modular coincide. Example~\ref{eg:weightNeeded} shows that interesting things can happen in the non-autonomous 
case, which do not appear at all when $\phi$ is independent of $x$. 

\begin{eg}
Let $\phi(x,t):= t$ and consider the corresponding functions $V_1$ and $\rho_{V,1}$. 
Then $\phi^*(x,t)=\infty \chi_{(1,\infty)}(t)$ so that $\rho_{\phi^*}(w)<\infty$ 
if and only if $w\le 1$ almost everywhere, in which case $\rho_{\phi^*}(w)=0$. Hence  
$V_1(u) = \rho_{V,1}(u) = |Du|(\Omega) =V(u,\Omega)$. 
\end{eg}

\begin{eg}\label{eg:weightNeeded}
Let $\phi(x,t):= \frac1{p(x)}t^{p(x)}$ for $p:\R\to[1,\infty)$. 
Then $\phi^*(x,t)=\frac1{p'(x)}t^{p'(x)}$ when $p(x)>1$ and 
$\phi^*(x,t)=\infty \chi_{(1,\infty)}(t)$ when $p(x)=1$. Consider the Heaviside function 
$h=\chi_{(0,\infty)}$ so that $Dh=\delta_{\{0\}}$, the Dirac measure. Now 
\[
\rho_{V,\phi}(h)
= 
\sup\big\{ w(0) - \rho_{\phi^*}(|w|) \,|\, w \in C^1_0(\Omega)\big\}.
\]
Since $w$ is continuous, there exists for every $\epsilon\in (0,w(0))$ a number $\delta>0$ such that 
\[
\rho_{\phi^*}(w) \ge \int_{-\delta}^\delta \frac1{p'(x)}(w(0)-\epsilon)^{p'(x)}.
\]
Suppose first that $p(x):=1+\frac {c_{\log}}{\log(1/|x|)}$ for small $|x|$. Then 
$p'(x)=\frac {\log(1/|x|)}{c_{\log}} + 1$ and 
\[
(w(0)-\epsilon)^{p'(x)} = |x|^{-\frac{\log(w(0)-\epsilon)}{c_{\log}}} (w(0)-\epsilon).
\]
Hence the previous integral converges if $\log w(0)<c_{\log}$ and diverges if $\log w(0)>c_{\log}$ 
(when $\epsilon\to 0^+$). 
On the other hand, we can choose $w$ such that $0\le w\le w(0)\chi_{[-\delta, \delta]}$. From this 
we see that $\inf_w \rho_{\phi^*}(w) = 0$ when $\log w(0)<c_{\log}$. It follows that 
\[
\rho_{V,\phi}(h)
= 
e^{c_{\log}}.
\]
In the same way we can show that $\rho_{V,\phi}(h)=1$ if $p(x):=1+|x|^\alpha$ for some $\alpha>0$. 
This example shows that $\rho_{V,\phi}(h)$ depends on the behavior of the exponent in 
a neighborhood of $0$, even though the support of the derivative is only $\{0\}$. 
\end{eg}

\begin{rem}\label{rem:restrictedTestFunction}
If $\rho_{\phi^*}(|w|)=\infty$, then 
$\int_\Omega u \div w - \phi^*(x, |w|)\, dx = - \infty$
since $\int_\Omega u \div w \,dx $ is finite as $u\in L^1(\Omega)$ and $w \in C^1_0(\Omega; \Rn)$.
Testing with $w\equiv 0$, we see that the supremum in $\rho_{V,\phi}$ is always non-negative. 
Therefore test-functions $w$ with $\rho_{\phi^*}(|w|)=\infty$ can be omitted and we obtain 
the alternative, equivalent formulation 
\[
\rho_{V,\phi}(u)
=
\sup\bigg\{ \int_\Omega u \div w - \phi^*(x, |w|)\, dx \,\Big|\, w \in C^1_0(\Omega; \Rn), \rho_{\phi^*}(|w|)<\infty 
\bigg\}.
\]
Note that $\rho_{\phi^*}(|w|)<\infty$ does not follow from $w \in C^1_0(\Omega; \Rn)$ as 
$\phi^*$ does not satisfy \adec{}.
\end{rem}

\begin{rem}
In our definition we use test-functions from $C^1_0(\Omega;\Rn)$. This corresponds 
to the definition of the usual $\BV$-space. An alternative in duality formulations (e.g.\ \cite{AmeGZ14, CheLR06}) 
is $C^1(\Omega;\Rn)$, which means that the boundary values of the function $u$ also influence the norm and 
modular. The restriction $|w|\le 1$ carries over nicely to the boundary and leads to a boundary term in 
$L^1(\partial \Omega)$. This is not the case with $\rho_{\phi^*}(|w|)<\infty$. It seems that 
an additional boundary term for $w$ of fractional Sobolev space-type 
is needed in $\rho_{V,\phi}$ if we want to obtain appropriate 
boundary values in the generalized Orlicz case. This remains a problem for future research. 
\end{rem}

Let $w \in C^1_0(\Omega; \Rn)$ with $|w|\le 1$. Since $\Omega$ is bounded and \azero{} 
for $\phi$ implies \azero{} for $\phi^*$, we find that $w \in L^{\phi^*}(\Omega)$ and 
$\|w\|_{\phi^*} \le c\, \|w\|_\infty \le c$, see Corollary 3.7.10 in \cite{HarH19}. 
By the definition of $V_\phi$, 
\[
\int_\Omega u \div w \, dx =  \|w\|_{\phi^*} \int_\Omega u \div \tfrac w{ \|w\|_{\phi^*}} \, dx 
\le \|w\|_{\phi^*} V_\phi(u)
\]
and taking supremum over all such $w$, we obtain that $V(u,\Omega)\le c V_\phi(u)$.
Since $\Omega$ is bounded and $\phi$ satisfies \azero{}, we have $L^\phi(\Omega) \hookrightarrow L^1(\Omega)$ by 
\cite[Corollary 3.7.9]{HarH19}. Thus $\BV^\phi(\Omega) \hookrightarrow \BV(\Omega)$ provided $\phi$ satisfies \azero{}.

%We next show that $\|\cdot\|_{\BV^\phi}$ is a norm not just for $\phi\in \Phic(\Omega)$ 
%but also for $\phi\in \Phiw(\Omega)$. 

\begin{lem}
If $\phi \in \Phiw(\Omega)$, then $V_\phi$ is a seminorm 
and $\|\cdot\|_{\BV^\phi}$ is a quasinorm in $\BV^\phi(\Omega)$. 
Moreover, if $\phi \in \Phic(\Omega)$, then $\|\cdot\|_{\BV^\phi}$ is a norm.
\end{lem}

\begin{proof}
%If $\|u\|_{\BV^\phi}=0$, then $\|u\|_\phi=0$ and thus $f=0$ a.e. 
The homogeneity property $V_\phi(au)= |a|V_\phi(u)$ is clear.
Let us show that $V_\phi$ satisfies the triangle inequality. 
If $u, v \in L^1(\Omega)$, then
\[
\int_\Omega (u + v) \div w \, dx
= 
\int_\Omega u \div w \, dx + \int_\Omega v \div w \, dx \le V_\phi(u) + V_\phi(v)
\]
for $w \in C^1_0(\Omega; \Rn)$ with $\|w\|_{\phi^*}\le 1$. 
By taking the supremum over such $w$ we have 
\[
V_\phi(u+v) \le V_\phi(u) + V_\phi(v). 
\]
Note that $\|\cdot \|_\phi$ is a quasinorm if $\phi \in \Phiw(\Omega)$, and a norm if $\phi \in \Phic(\Omega)$.
Combining these two results, we obtain 
the (quasi)triangle inequality for the sum that is ${\|\cdot\|}_{\BV^\phi(\Omega)}$.
These properties also imply that $\BV^\phi(\Omega)$ is a vector space. 
\end{proof}

From the next lemma it follows that the sum $\rho_\phi + \rho_{V,\phi}$ is a quasi-semimodular, and a 
quasimodular if $\phi$ satisfies \adec{}. Note that the convexity of $\phi$ is  not required.

\begin{lem}\label{lem:pseudo-modular}
If $\phi\in\Phiw(\Omega)$, then $\rho_{V,\phi}$ is a left-continuous semimodular in $L^1(\Omega)$. 
\end{lem}
\begin{proof}
Since $w=0$ is a possible test function, we see that $\rho_{V,\phi}\ge 0$. 
If $u=0$ a.e., then the integrand in $\rho_{V,\phi}$ is the non-positive function $-\phi^*(x, |w|)$,
so that $\rho_{V,\phi}(0)\le 0$. Thus $\rho_{V,\phi}(0)=0$ 
and property (1) from the definition of semimodular holds.
If $w$ is a test function, then so is $-w$ 
and hence $\rho_{V,\phi}(-u)=\rho_{V,\phi}(u)$. Thus property (3) holds, as well.

To show that $\lambda \mapsto \rho(\lambda x)$ is increasing we
let $\lambda \in (0, 1)$ and $w \in C^1_0(\Omega; \Rn)$.
Since $\phi^*$ is increasing, 
\[
\int_\Omega \lambda u \div w - \phi^*(x, |w|)\, dx 
\le \int_\Omega u \div(\lambda w) - \phi^*(x, |\lambda w|)\, dx 
\le \rho_{V,\phi}(u),
\]
as $\lambda w \in C^1_0(\Omega; \Rn)$. Taking the 
supremum over $w$, we get $\rho_{V,\phi}(\lambda u) \le \rho_{V,\phi}(u)$.

Let us prove that $\rho_{V,\phi}$ is convex. 
Let $u, v \in L^1(\Omega)$, $\theta \in (0, 1)$ and $w\in C^1_0(\Omega; \Rn)$. Then
\[
\begin{split}
&\int_\Omega (\theta u + (1-\theta) v) \div w - \phi^*(x, |w|)\, dx\\
& \qquad = \theta\int_\Omega u \div w - \phi^*(x, |w|)\, dx
+(1-\theta)\int_\Omega v \div w - \phi^*(x, |w|)\, dx\\
&\qquad \le \theta \rho_{V,\phi}(u) + (1-\theta) \rho_{V,\phi}( v).
\end{split}
\] 
The claim follows when we take the supremum over $w\in C^1_0(\Omega; \Rn)$. 

Finally, we show that $\rho$ is left-continuous. 
Since $\lambda\mapsto \rho_{V,\phi}(\lambda u)$ is increasing, 
$\rho_{V,\phi}(\lambda u) \le \rho_{V,\phi}(u)$ for $\lambda \in(0, 1)$.
We next consider the opposite inequality at the limit.
Let first $\rho_{V,\phi}(u)<\infty$ and fix $\epsilon>0$. 
By the definition of $\rho_{V,\phi}$ and Remark~\ref{rem:restrictedTestFunction} there exists a test function 
$w\in C^1_0(\Omega; \Rn)$ with $\rho_{\phi^*}(|w|)<\infty$ such that
\[
\int_\Omega u \div w \, dx 
\ge 
\rho_{V,\phi}(u) - \epsilon + \rho_{\phi^*}(|w|). 
\]
Multiplying the inequality by $\lambda\in (0,1)$ and subtracting $\rho_{\phi^*}(|w|)$, 
we obtain that 
\[
\rho_{V,\phi}(\lambda u) 
\ge 
\lambda (\rho_{V,\phi}(u) - \epsilon) + (\lambda-1)\rho_{\phi^*}(|w|). 
\]
Hence 
\[
\lim_{\lambda\to 1^-}\rho_{V,\phi}(\lambda u) 
\ge 
\rho_{V,\phi}(u) - \epsilon. 
\]
The claim follows from this as $\epsilon\to 0^+$. The case $\rho_{V,\phi}(u)=\infty$ is proved similarly, 
we only need to replace $\rho_{V,\phi}(u) - \epsilon$ by $\frac1\epsilon$. 
\end{proof}

From the previous lemma it follows that $\rho_{V,\phi}$ defines a seminorm by the 
Luxenburg method (Definition~\ref{def:modularSpace}); for a proof see \cite{HarHJ_pp}.
We next show that this seminorm is comparable to $V_\phi$. 

\begin{lem}\label{lem:equivalence}
If $\phi\in\Phiw(\Omega)$ and $u\in \BV^\phi(\Omega)$, then
\[
\| u\|_{\rho_{V,\phi}} \le V_\phi(u) \le 2 \| u\|_{\rho_{V,\phi}}. 
\]
\end{lem}
\begin{proof}
If $V_\phi(u)=0$, then $\rho_{V,\phi}(\frac u\lambda)=0$ for every $\lambda>0$ so that $\| u\|_{\rho_{V,\phi}} = 0$. The case $V_\phi(u)=\infty$ is excluded by the assumption $u\in \BV^\phi(\Omega)$.
%\comment{
%Peter does not see why this is needed for the proof:\\
%If $\| u\|_{\rho_{V,\phi}} =0$, then $\rho_{V,\phi}(\frac u\lambda)=0$ for every $\lambda>0$. This yields for every 
%$w \in C^1_0(\Omega; \Rn)$ and every $\lambda>0$ that
%\[
%\int_\Omega u \div w -  \lambda \phi^*(x, |w|) \, dx =0.
%\]
%Let $(\lambda_i)$ be a positive sequence converging to $0$. Monotone convergence theorem yields that
%\[
%\int_\Omega u \div w  = \lim_{i \to \infty} \int_\Omega u \div w -  \lambda_i \phi^*(x, |w|) \, dx =0
%\]
%for every 
%$w \in C^1_0(\Omega; \Rn)$, and hence  $V_\phi(u)=0$.
%}
Since the claim is homogeneous, the case $V_\phi(u)\in (0,\infty)$ reduces to $V_\phi(u)=1$. 
By the definition of $V_\phi$, it then follows that 
\[
\int_\Omega u \div w \, dx
\le 
V_\phi(u) \|w\|_{\phi^*}
= 
\|w\|_{\phi^*}
\le 
1+\rho_{\phi^*}(|w|);
\]
the last step is a general property of the Luxemburg norm, see \cite[Corollary~2.1.15]{DieHHR11}. Thus 
\[
\rho_{V,\phi}(u) 
%\le 
%\sup_{w \in C^1_0(\Omega; \Rn)}\big\{ \|w\|_{\phi^*} - \rho_{\phi^*}(|w|) \big\}
\le 
\sup_{w \in C^1_0(\Omega; \Rn), \rho_{\phi^*}(|w|)<\infty}\big(1+\rho_{\phi^*}(|w|) - \rho_{\phi^*}(|w|) \big)
= 1,
\]
and so $\|u\|_{\rho_{V,\phi}}\le 1$. This concludes the proof of the first inequality. 

We next establish the opposite inequality $2\|u\|_{\rho_{V,\phi}} \ge 1$, which is equivalent to 
$\rho_{V,\phi}(\frac{2u}\lambda)\ge 1$ for every $\lambda<1$ 
. Since $\rho_{\phi^*}(|w|)\le 1$ when $\|w\|_{\phi^*}\le 1$, 
we conclude that
\begin{align*}
\rho_{V,\phi}(\tfrac{2u}\lambda)
&\ge
\sup\bigg\{ \int_\Omega \tfrac{2u}\lambda \div w - \phi^*(x, |w|)\, dx \,\Big|\, w \in C^1_0(\Omega; \Rn), \|w\|_{\phi^*}\le 
1\bigg\}\\
&\ge
\tfrac 2\lambda V_\phi(u) 
- 1 > 1. \qedhere
\end{align*}
\end{proof}

The following result is the counterpart of Theorem 5.2 in \cite{EvaG92}, see also Theorem 1.9 in \cite{Giusti}.

\begin{lem}[Weak lower semicontinuity]\label{lem:sequence-in-BV}
%\marginpar{Why is $k\in\N$ mentioned here but not in other similar places?}
Let $\phi \in \Phiw(\Omega)$, $u, u_k \in L^1(\Omega)$ with
$u_k \rightharpoonup u$ in $L^1(\Omega)$. Then 
\[
V_\phi(u) \le \liminf_{k\to \infty} V_\phi(u_k)
\qquad\text{and}\qquad
\rho_{V,\phi}(u) \le \liminf_{k\to \infty} \rho_{V,\phi}(u_k).
\]
\end{lem}
\begin{proof}
If $w \in C^1_0(\Omega; \Rn)$, then $\div w \in L^\infty(\Omega)$ 
and weak convergence in $L^1(\Omega)$ with $\|w\|_{\phi^*}\le 1$ give 
\[
\int_\Omega u \div w \, dx 
= 
\lim_{k \to \infty} \int_\Omega u_k \div w \, dx\le \liminf_{k \to \infty} V_\phi(u_k).
\]
The first inequality of the claim follows by taking the supremum over all such $w$. Subtracting 
$\rho_{\phi^*}(|w|)<\infty$ from both sides of the equality similarly gives the second inequality. 
\end{proof}

%%%%%%%%%%%%%%%%%%%%%%%%%%%%%%%%%%%%%%%%%%%%%%%%%%%%%%%%
%%%%%%%%%%%%%%%%%%%%%%%%%%%%%%%%%%%%%%%%%%%%%%%%%%%%%%%%
%%%%%%%%%%%%%%%%%%%%%%%%%%%%%%%%%%%%%%%%%%%%%%%%%%%%%%%%
\section{Approximation properties of the dual norm}
\label{sect:approximation}

In this section we prove a compactness-type property of $\BV^\phi$ and estimate 
the $V_\phi$-norm of $W^{1,1}_\loc$-functions by $\|\nabla u\|_\phi$. 
We first connect the norm $V_\phi$ with the associate space $(L^{\phi^*})'$-norm of the gradient. 
One crucial difference between these norms is that 
in the associate space norm we test with functions in $L^{\phi^*}$ whereas in $V_\phi$ the test functions 
are smooth. Thus some approximation is needed, but we cannot use density in $L^{\phi^*}(\Omega)$ 
since $\phi^*$ is not, in general, doubling. 
We start with a property of lower semicontinuous functions. Although the result is known, we did not find a 
reference for these exact properties of the approximating functions, so we provide a proof for completeness.

\begin{lem}\label{lem:lsc-approx}
Let $f:\Omega\to [0,\infty]$ be lower semicontinuous. Then there exist functions 
$w_i\in C^1_0(\Omega)$ with $0\le w_i\le f$ and $w_i\to f$.
\end{lem}
\begin{proof}
We first define 
\[
f_i := \sum_{k=1}^\infty 2^{-i} \chi_{\{f>2^{-i}k\}}.
\]
If $f(x)\in (2^{-i}k, 2^{-i}(k+1)]$, then $f_i(x)= 2^{-i}k$. Hence $0\le f_i\le f$ and $f_i\nearrow f$. 
Thus it suffices to approximate $f_i$ and use a diagonal argument. 
Since $\{f>2^{-i}k\}$ is open, we can find non-negative functions $w_j^{k,i}\in C^1_0(\{f>2^{-i}k\})$ 
with $w_j^{k,i} \nearrow \chi_{\{f>2^{-i}k\}}$ as $j\to\infty$. Set 
\[
w_j^i := \sum_{k=1}^j 2^{-i} w_j^{k,i}. 
\]
Since each sum is finite, $w_j^i\in C^1_0(\Omega)$. Furthermore, 
$w_j^i \nearrow f_i$ as $j\to\infty$. 
\end{proof}

In Theorem~\ref{thm:BV-gradient}(1) we assume that $C^1_0(\Omega; \Rn)$ is dense in $L^{\phi^*}(\Omega; \Rn)$.
If $\phi^*$ satisfies \azero{} and \adec{}, then this holds by \cite[Theorem~3.7.15]{HarH19}.
Furthermore, $\phi^*$ satisfies these conditions 
if and only if $\phi$ satisfies \azero{} and \ainc{}. 
In other words, this is exactly the opposite of the linear growth case that 
we are interested in. 
However, in this case we can give an exact formula for the variation $V_\phi$ in terms of the norm of the 
associate space.

The case when $\phi^*$ does not satisfy \adec{} is more interesting
and involves the technical difficulties that we expect with $\BV$-type spaces. 
Now we need to approximate 
not the test function but the function itself so the regularity of 
$\phi$ matters.

\begin{thm}\label{thm:BV-gradient}
Let $\phi \in \Phiw(\Omega)$ and $u \in W^{1, 1}_\loc(\Omega)$. 
Then $V_\phi(u) \le \| \nabla u \|_{(L^{\phi^*}(\Omega))'} $.
\begin{enumerate}
\item[(1)]
If $C^1_0(\Omega; \Rn)$ is dense in $L^{\phi^*}(\Omega; \Rn)$, 
then $V_\phi(u) = \| \nabla u \|_{(L^{\phi^*}(\Omega))'}$.
\item [(2)]
If $\phi$ satisfies \azero{}, \aone{} and \adec{}, 
%and $u \in C^\infty(\Omega)$ or $u \in W^{1,1}(\Omega) \cap L^\phi(\Omega)$, 
then $V_\phi(u) \approx \| \nabla u \|_\phi$.
\end{enumerate}
\end{thm}
\begin{proof}
Since $u \in W^{1, 1}_\loc(\Omega)$, it follows from the 
definition of $V_\phi$ and integration by parts that 
\begin{equation}\label{eq:integrationByParts}
V_\phi(u)=\sup\bigg\{ \int_\Omega \nabla u \cdot w \, dx \,\Big|\, w \in C^1_0(\Omega; \Rn), \| w \|_{\phi^*}\le 1 
\bigg\}. 
\end{equation}
The definition of the associate space norm implies that 
\[
\int_\Omega \nabla u \cdot w \, dx
\le 
\int_\Omega |\nabla u|\, |w| \, dx
\le 
\| \nabla u \|_{(L^{\phi^*}(\Omega))'} \|w\|_{L^{\phi^*}(\Omega)}.
\]
Taking the supremum over $w\in C^1_0(\Omega; \Rn)$ with $\|w\|_{L^{\phi^*}(\Omega)}\le 1$, we conclude that 
$V_\phi(u) \le \| \nabla u \|_{(L^{\phi^*}(\Omega))'}$. 
%It remains to prove 
%that $V_\phi(u) \gtrsim \| \nabla u \|_{(L^{\phi^*}(\Omega))'}$.

Under assumption (1), we next show the opposite inequality, 
$ \| \nabla u \|_{(L^{\phi^*}(\Omega))'} \le V_\phi(u)$. 
Let $ w \in L^{\phi^*}(\Omega; \Rn)$ with $\|w\|_{\phi^*}=1$ and let $(w_j)$ be a sequence from 
$C^1_0(\Omega; \Rn)$ with $w_j \to w$ in $L^{\phi^*}(\Omega; \Rn)$ and pointwise a.e. 
Since also $w_j/\|w_j\|_{\phi^*} \to w$ in $L^{\phi^*}(\Omega; \Rn)$, we may assume that 
$\|w_j\|_{\phi^*}=1$. By Fatou's Lemma, 
\[
\liminf_{j \to \infty} \int_\Omega \nabla u \cdot w_j \, dx \ge \int_\Omega \nabla u \cdot w \, dx,
\]
so it follows from \eqref{eq:integrationByParts} that 
\[
V_\phi(u)\ge \sup\bigg\{ \int_\Omega \nabla u \cdot w \, dx \,\Big|\, w \in L^{\phi^*}(\Omega; \Rn), \| 
w\|_{\phi^*}\le 
1 \bigg\}.
\]
Let $h\in L^{\phi^*}(\Omega)$.
We set $w:=\frac{\nabla u}{|\nabla u|} h$ if $|\nabla u|\neq 0$ and $0$ otherwise.
This gives 
\[
V_\phi(u)\ge \sup\bigg\{ \int_\Omega |\nabla u| \, h \, dx \,\Big|\, h \in L^{\phi^*}(\Omega), \| h\|_{\phi^*}\le 1 
\bigg\} = \| \nabla u \|_{(L^{\phi^*}(\Omega))'}. 
\]
Hence $V_\phi(u) = \| \nabla u \|_{(L^{\phi^*}(\Omega))'}$ and the proof of (1) is complete.

Then we prove (2).
%Let $h\in L^{\phi^*}(\Omega;\Rn)$ with $\|h\|_{\phi^*}\le 1$. Since $\phi^*$ satisfies 
%\azero{}, \aone{} and \ainc{}, $M$ 
%is bounded in $L^{\phi^*}(\Omega)$, with constant $c_M>0$ \cite{Has15}. 
%The function $\tilde h:=\frac{(Mh_1, \ldots, Mh_n)}{n c_M}\in L^{\phi^*}(\Omega;\Rn)$ has lower semi-continuous 
%components, $\|\tilde h\|_{\phi^*}\le 1$ and $\tilde h\ge \frac h{n c_M}$. 
%By Lemma~\ref{lem:lsc-approx}, we can find $h^i\in C^1_0(\Omega;\Rn)$ with $h_j^i\to \tilde h_j$ 
%and $0\le h_j^i\le \tilde h_j$ for every $j=1, \ldots, n$. By dominated convergence, with $|\nabla u| \,|\tilde h|$ as a majorant, we find that 
%\[
%V_\phi(u)
%\ge
%\lim_{i\to \infty}\int_\Omega \nabla u\cdot h_i \, dx 
%=
%\int_\Omega \nabla u\cdot \tilde h \, dx
%\ge
%\frac1{n c_M}\int_\Omega \nabla u\cdot h \, dx.
%\]
%Since $h\in L^{\phi^*}(\Omega; \Rn)$ is arbitrary, we can use the same $w:=\frac{\nabla u}{|\nabla u|} h$ 
%trick as before to conclude that 
%\[
%V_\phi(u)
%\ge
%\frac1{n c_M}
%\sup\bigg\{ \int_\Omega |\nabla u|\, h \, dx \,\Big|\, h \in L^{\phi^*}(\Omega), \|h \|_{\phi^*}\le 1 \bigg\}
%=
%\frac{ \| \nabla u\|_{(L^{\phi^*}(\Omega))'} }{n c_M}. \qedhere
%\]
%
%
Fix $h\in C^1_0(\Omega)$. Since $w_{\epsilon,\delta}:=(\frac{\nabla u}{|\nabla u|+\epsilon})* \eta_\delta$ is 
bounded and converges to $\frac{\nabla u}{|\nabla u|+\epsilon}$ in $L^1$ and a.e.\ as $\delta\to 0^+$, 
it follows by $L^1$-convergence and dominated convergence with majorant $|\nabla u|\, |h|$ that 
\[
\lim_{\epsilon\to 0^+}\lim_{\delta\to 0^+}\int_\Omega \nabla u \cdot (w_{\epsilon,\delta} h) \, dx
=
\lim_{\epsilon\to 0^+}\int_\Omega \frac{|\nabla u|^2}{|\nabla u|+\epsilon}\, h \, dx
=
\int_\Omega |\nabla u|\, h \, dx.
\]
Since $w_{\epsilon,\delta}h\in C^1_0(\Omega;\Rn)$, this and \eqref{eq:integrationByParts} imply that 
\[
%\sup\bigg\{ \int_\Omega \nabla u \cdot w \, dx \,\Big|\, w \in C^1_0(\Omega; \Rn), \| w \|_{\phi^*}\le 1 
%\bigg\}
V_\phi(u)
\ge
\sup\bigg\{ \int_\Omega |\nabla u|\, h \, dx \,\Big|\, h \in C^1_0(\Omega), \| h \|_{\phi^*}\le 1 \bigg\}. 
\]
Let $g\in L^{\phi^*}(\Omega)$ with $\| g \|_{\phi^*}\le 1$. Since $\phi^*$ satisfies 
\azero{}, \aone{} and \ainc{}, the Hardy--Littlewood maximal operator $M$ 
is bounded in $L^{\phi^*}(\Omega)$, with some constant $c_M>0$ \cite{Has15}. 
The function $\tilde g:=\frac{Mg}{c_M}\in L^{\phi^*}(\Omega)$ is lower semi-continuous, 
$\| \tilde g \|_{\phi^*}\le 1$ and $\tilde g\ge \frac g{c_M}$. 
By Lemma~\ref{lem:lsc-approx}, we can find $h_i\in C^1_0(\Omega)$ with $h_i\to \tilde g$ 
and $0\le h_i\le \tilde g$. By dominated convergence, with $|\nabla u| \,\tilde g$ as a majorant, we find that 
\[
V_\phi(u)
%\ge
%\sup\bigg\{ \int_\Omega |\nabla u| h \, dx : h \in C^1_0(\Omega), \| h \|_{\phi^*}\le 1 \bigg\}
\ge
\lim_{i\to \infty}\int_\Omega |\nabla u|\, h_i \, dx 
=
\int_\Omega |\nabla u|\, \tilde g \, dx
\ge
\frac1{c_M}\int_\Omega |\nabla u|\, g \, dx.
\]
Since $g$ is arbitrary, this implies that 
\[
V_\phi(u)
\ge
\frac1{c_M}
\sup\bigg\{ \int_\Omega |\nabla u|\, g \, dx \,\Big|\, g \in L^{\phi^*}(\Omega), \|g \|_{\phi^*}\le 1 \bigg\}
=
\frac{ \| \nabla u\|_{(L^{\phi^*}(\Omega))'} }{c_M}.
%\gtrsim
%\frac{ \| \nabla u\|_{L^\phi(\Omega)} }{c_M}. 
\]
By Theorem~3.4.6 and Proposition~2.4.5 of \cite{HarH19}, 
$\| \nabla u \|_{(L^{\phi^*}(\Omega))'} \approx \| \nabla u\|_\phi$, 
so the proof of (2) is complete. 
\end{proof}

The next lemma is the counterpart of \cite[Theorem 5.3]{EvaG92} and \cite[Theorem 1.17]{Giusti}, 
albeit with an extra constant $c_\phi$. The extra constant is expected, since we assume only 
\aone{}, cf.\ Example~\ref{eg:weightNeeded} and Proposition~\ref{prop:modularDensity}.

\begin{lem}[Approximation by smooth functions]\label{lem:density}
Assume that $\phi \in \Phiw(\Omega)$ satisfies \azero{}, \aone{} and \adec{}. 
Then there exists $c_\phi\ge 1$ 
such that for every $u \in L^\phi(\Omega)$ we can find $u_k \in C^\infty(\Omega)$ with 
\[
u_k \to u\text{ in }L^\phi(\Omega)
\quad\text{and}\quad
V_\phi(u) \le \lim_{k \to \infty} V_\phi(u_k)\le c_\phi V_\phi(u).
\]
If additionally $u\in W^{1,p}(\Omega)$ or $u \in L^p(\Omega)$, $p\in [1,\infty)$, 
then the sequence can be chosen with $u_k \to u$ in $ W^{1,p}(\Omega)$ or $L^p(\Omega)$ as well.
\end{lem}
\begin{proof}
For $k\in\Z_+$, we define
\[
U_k := \bigg\{x \in \Omega \,\Big|\, \dist(x, \partial \Omega) > \frac1{m+k} \bigg\},
\]
where $m>0$ is chosen such that $U_1$ is non-empty. 
Set $V_1:=U_2$ and $V_k := U_{k+1} \setminus \overline{U_{k-1}}$ for $k\ge 2$. 
Let $(\xi_k)$ be a partition of unity subordinate to $(V_k)$, i.e.\ 
$\xi_k \in C^\infty_0 (V_k)$, $0 \le \xi_k \le 1$ and $\sum_{k=1}^\infty \xi_k=1$ for all $x \in \Omega$.

Let $\epsilon>0$ and let $\eta$ be the standard mollifier. Choose $\epsilon_k\in (0,\epsilon)$ so small that
$\supp(\eta_{\epsilon_k} * (u \xi_k)) \subset V_k$, 
\begin{align}\label{eq:density}
\| \eta_{\epsilon_k}* (u\xi_k) -u\xi_k \|_\phi \le \frac{\epsilon}{2^k} 
\qquad\text{and}\qquad
\| \eta_{\epsilon_k}* (u \nabla \xi_k) -u \nabla \xi_k \|_\phi \le \frac{\epsilon}{2^k};
\end{align}
the last conditions are possible by \azero{}, \aone{} and \adec{} 
since $u\in L^\phi(\Omega)$ and $\xi_k, |\nabla \xi_k|\in L^\infty(\Omega)$ \cite[Theorem~4.4.7]{HarH19} .
Let us define
\[
u_\epsilon := \sum_{k=1}^\infty \eta_{\epsilon_k}* (u\xi_k).
\]
In a neighborhood of each point there are at most three non-zero terms in the sum, 
hence $u_\epsilon \in C^\infty (\Omega)$. 

Since $\|\cdot\|_\phi$ is equivalent to a norm, it satisfies a countable quasitriangle inequality \cite[Corollary~3.2.5]{HarH19}. 
Using $u = \sum_{k=1}^\infty \xi_ku$ and \eqref{eq:density} with this inequality, we find that 
\[
\begin{split}
\|u_\epsilon -u\|_\phi &\le \Big\|\sum_{k=1}^\infty (\eta_{\epsilon_k}* (u\xi_k) - \xi_ku) \Big\|_\phi
\lesssim \sum_{k=1}^\infty \| \eta_{\epsilon_k}* (u\xi_k) -u\xi_k \|_\phi 
\le
\sum_{k=1}^\infty\frac{\epsilon}{2^k} = \epsilon.
\end{split}
\]
Thus $u_\epsilon \to u$ in $L^\phi(\Omega)$ and so 
Lemma~\ref{lem:sequence-in-BV} yields
\[
V_\phi(u) \le \liminf_{\epsilon\to 0^+} V_\phi(u_\epsilon). 
\]
If we assume $u \in L^p(\Omega)$ or $|\nabla u|\in L^p(\Omega)$, then we can add to \eqref{eq:density} also the 
requirement 
\[
\| \eta_{\epsilon_k}* (u\xi_k) -u\xi_k \|_p \le \frac{\epsilon}{2^k}
\quad\text{or}\quad
\| \eta_{\epsilon_k}* (\nabla u \,\xi_k) -\nabla u \,\xi_k \|_p \le \frac{\epsilon}{2^k}
\]
and estimate in the same way $\|u_\epsilon - u\|_p\le \epsilon$ 
or $\|\nabla u_\epsilon - \nabla u\|_p\le \epsilon$. 
Thus also $u_\epsilon \to u$ in $L^p(\Omega)$ or 
$W^{1,p}(\Omega)$, as claimed.

%We then show that $V_\phi(u_\epsilon)\le LV_\phi(u) + c\epsilon$. 
%Let $w \in C^1_0(\Omega; \Rn)$ with $\|w\|_{\phi*}\le 1$. 
%First we show that $\int_\Omega u_\epsilon \div w \, dx = \sum_{k=1}^\infty \int_\Omega (\eta_{\epsilon_k}*(u\xi_k)) \div w \, dx $, and for that we use the dominated convergence theorem. We estimate
%\[
%|u_\epsilon| \le \sum_{k=1}^\infty |\eta_{\epsilon_k} * (u \xi_k)|\le  \sum_{k=1}^\infty\eta_{\epsilon_k} * |u\xi_k|.
%\]
%\comment{Since $w \in C^1_0(\Omega; \Rn)$, actually each sum is finite! So this is not needed after all...}
%Since the functions $\eta_{\epsilon_k} * |u\xi_k|$ are non-negative we obtain
%\[
%\begin{split}
%\int_\Omega \sum_{k=1}^\infty \eta_{\epsilon_k} * |u\xi_k| \, dx &= \sum_{k=1}^\infty \int_\Omega  \eta_{\epsilon_k} * |u\xi_k| \, dx\\
%&=\sum_{k=1}^\infty c\int_\Omega |u| \xi_k \, dx =  c \int_\Omega \sum_{k=1}^\infty |u| \xi_k \, dx =  c \| u \|_{L^1(\Omega)}
%\end{split}
%\]
%so that $\sum_{k=1}^\infty \eta_{\epsilon_k} * |u\xi_k| \in L^1(\Omega)$.  
%Since $w \in C^1_0(\Omega;\Rn)$, $\div w$ is bounded.
%Now we can use the dominated convergence theorem to obtain
%\[
%\int_\Omega u_\epsilon \div w \, dx = \lim_{m \to \infty }  \int_\Omega \sum_{k=1}^m  (\eta_{\epsilon_k}*(u \xi_k)) \div w \, dx
%= \lim_{m \to \infty }  \sum_{k=1}^m  \int_\Omega  (\eta_{\epsilon_k}*(u \xi_k)) \div w \, dx.
%\]

Fix $w\in C^1_0(\Omega;\Rn)$ with $\|w\|_{\phi^*}\le 1$. Since $w$ has a compact support in $\Omega$, 
only finitely many of the terms $(\eta_{\epsilon_k}*(u\xi_k)) \div w$ are non-zero. 
Thus the sums below are really finite and can be interchanged with integrals and derivatives.
Using the definition of $u_\epsilon$, Fubini's Theorem in the convolution, the product rule 
and $\sum_{k=1}^\infty \nabla \xi_k = \nabla \sum_{k=1}^\infty \xi_k = \nabla 1=0$, 
we conclude that 
\[
\begin{split}
\int_\Omega u_\epsilon \div w \, dx 
&= \sum_{k=1}^\infty \int_\Omega (\eta_{\epsilon_k}*(u\xi_k)) \div w \, dx 
= \sum_{k=1}^\infty \int_\Omega (u\xi_k) \div (\eta_{\epsilon_k}* w) \, dx \\
&= \sum_{k=1}^\infty \int_\Omega u \div (\xi_k(\eta_{\epsilon_k}* w)) \, dx 
- \sum_{k=1}^\infty \int_\Omega u \nabla \xi_k \cdot (\eta_{\epsilon_k}* w) \, dx \\
&= \underbrace{ \sum_{k=1}^\infty \int_\Omega u \div (\xi_k(\eta_{\epsilon_k}* w)) \, dx}_{=: I} - 
\underbrace{ \sum_{k=1}^\infty \int_\Omega w \cdot (\eta_{\epsilon_k}*(u \nabla \xi_k) - u\nabla \xi_k)\,dx}_{=: 
II}\,.
\end{split}
\]
For $II$ we obtain by Hölder's inequality and \eqref{eq:density} that
\[
|II| 
\lesssim 
\sum_{k=1}^\infty \|w\|_{\phi^*} \big\| \eta_{\epsilon_k}*(u \nabla \xi_k) - u\nabla \xi_k \big\|_{\phi}
\le 
\sum_{k=1}^\infty \frac{\epsilon}{2^k} 
= \epsilon .
\]
As $\sum_{k=1}^\infty \xi_k(\eta_{\epsilon_k}* w) \in C^1_0(\Omega; \Rn)$ is a viable 
test function (up to a constant), we obtain that 
\[
\begin{split}
|I| 
= 
\bigg|\int_\Omega u \div \bigg(\sum_{k=1}^\infty \xi_k\,\eta_{\epsilon_k}* w\bigg) \, dx \bigg|
&\le 
 V_\phi(u) 
\bigg\|\sum_{k=1}^\infty \xi_k\,\eta_{\epsilon_k}* w\bigg\|_{\phi^*}\\
&\lesssim 
V_\phi(u) \| Mw\|_{\phi^*}.
\end{split}
\]
Since $\phi^*$ satisfies \azero{}, \aone{} and \ainc{}, maximal operator $M$ 
is bounded in $L^{\phi^*}(\Omega)$ \cite{Has15}. So the estimates for $I$ and $II$ give 
\[
\Big| \int_\Omega u_\epsilon \div w \, dx \Big| \lesssim V_\phi(u) + \epsilon.
\]
Hence $V_\phi(u_\epsilon) \lesssim V_\phi(u) + \epsilon \to V_\phi(u)$ as $\epsilon\to 0^+$. 
By choosing a subsequence we ensure that $\lim_k V_\phi(u_k)$ exists. 
\end{proof}

%\begin{defn}\label{defn:John}
A bounded domain $\Omega\subset\Rn$ is a \emph{John domain} if there exist constants $0< \alpha \le \beta<\infty$ and 
a point $x_0 \in \Omega$ such that each point $x\in \Omega$ can be joined to $x_0$ by a rectifiable curve 
$\gamma:[0,\ell_\gamma] 
\to \Omega$ parametrized by arc length with $\gamma(0) = x$, $\gamma(\ell_\gamma) = x_0$, 
$\ell_\gamma\leq 
\beta\,,$ and
\[
t \leq \tfrac{\beta}{\alpha} \dist\big(\gamma(t), \partial \Omega \big) \quad \text{for all} \quad t\in[0, 
\ell_\gamma].
\]
%The point $x_0$ is called a John center of $\Omega$, and the ball $B(x_0, \dist(x_0, \partial \Omega)$ is called a 
%John ball of $\Omega$.
%\end{defn}
%
Examples of John domains include convex domains and domains with Lipschitz boundary, but also some 
domains with fractal boundaries such as the von Koch snowflake.
The next compactness-type result is the counterpart of Theorem 5.5 in \cite{EvaG92}, 
see also Theorem 1.19 in \cite{Giusti}.

\begin{thm}\label{thm:BV-compact-Lphi}
Let $\Omega \subset \Rn$ be a bounded John domain and $\phi \in \Phiw(\Omega)$ satisfy \azero{}, \aone{} and \adec{}. 
If $(u_k)$ is a sequence in $\BV^\phi(\Omega)$ with $ \sup_k \|u_k\|_{\BV^\phi} < \infty$, then there exists a 
subsequence $(u_{k_j})$ and $u \in \BV^\phi(\Omega)$ such that $u_{k_j} \to u$ in $L^\phi(\Omega)$ 
and $\|u\|_{\BV^\phi}\le \liminf \|u_{k_j}\|_{\BV^\phi}$.
\end{thm}

\begin{proof}
By Lemma~\ref{lem:density}, we may choose functions $ v_k \in C^\infty(\Omega) \cap \BV^\phi(\Omega)$ such that
\[
\|u_k - v_k\|_\phi < \tfrac1k \qquad\text{and}\qquad \sup_k V_{\phi}(v_k) < \infty.
\]
Theorem~\ref{thm:BV-gradient}(2) for $v_k\in C^\infty(\Omega)$ 
yields that $\sup_k \|\nabla v_k\|_\phi < \infty$, so the sequence is bounded 
in $W^{1,\phi}(\Omega)$. Since $\Omega$ is a John domain, 
the compact embedding $W^{1,\phi}(\Omega) \hookrightarrow \hookrightarrow L^\phi(\Omega)$ holds 
\cite{HarH19, HarHJ_pp}, 
and thus $(v_k)$ has a subsequence $(v_{k_j})$ converging to some $u$ in $L^\phi(\Omega)$. 
Therefore $\|u_{k_j} - v_{k_j}\|_\phi < \frac1{k_j}$ implies that also $u_{k_j} \to u$ in $L^\phi(\Omega)$ 
and, by Lemma~\ref{lem:sequence-in-BV}, $u \in \BV^\phi(\Omega)$.
\end{proof}

%%%%%%%%%%%%%%%%%%%%%%%%%%%%%%%%%%%%%%%%%%%%%%%%%%%%%%%%
%%%%%%%%%%%%%%%%%%%%%%%%%%%%%%%%%%%%%%%%%%%%%%%%%%%%%%%%
%%%%%%%%%%%%%%%%%%%%%%%%%%%%%%%%%%%%%%%%%%%%%%%%%%%%%%%%

\section{Explicit expression for the dual modular}
\label{sect:explicit}

In this section we derive a formula for the ``dual modular'' $\rho_{V,\phi}$ from Definition~\ref{defn:V-rho} in terms of $\rho_\phi$ of the derivative's absolutely continuous part and the singular part with weight 
given by the recession function 
\[
\phi'_\infty(x)= \limsup_{t\to \infty} \frac{\phi(x, t)} t.
\]
Throughout this section, we assume that $\phi\in \Phic(\Omega)$. 
Then $t\mapsto\frac{\phi(\cdot, t)} t$ is increasing and the limit superior is a limit. 
Moreover, if the derivative of $\phi$ with respect to $t$ exists, then 
it is increasing by convexity, so $\lim_{t\to \infty} \phi'(\cdot, t)$
exists and equals $\phi'_\infty$ by l'H\^{o}pital's rule.
The following lemma illustrates the significance of $\phi'_\infty$. 

In \cite[Section~3 and Example A.1]{HarHL08} it was shown that $\log$-H\"older continuity is not 
sufficient when working in $\BV^\px$. 
Similarly, the \aone{} condition (corresponding to $\log$-Hölder continuity) is not sufficient 
in the next results 
in view of Example~\ref{eg:weightNeeded}. Instead, we use the restricted \VA{} which corresponds 
to strong $\log$-H\"older continuity (Proposition~\ref{prop:strong}). 
Note that here we need the inequality 
at every point, since we will use the estimate with the singular measure $D^su$. 

\begin{lem}\label{lem:bound}
Let $\phi\in \Phic(\Omega)$ satisfy restricted \VA{}.
If $w\in C(\Omega)$ with $\rho_{\phi^*}(w)<\infty$, then $|w|\le \phi'_\infty$. 
\end{lem}
\begin{proof}
We assume that $w\ge 0$ to simplify notation. 
Suppose to the contrary that $w(x_0)>\phi'_\infty(x_0)$ for some point $x_0\in\Omega$. 
Since $\phi$ is convex, $t\mapsto \frac{\phi(x_0, t)} t$ is increasing and 
$\frac{\phi(x_0, t)} t\le \phi'_\infty(x_0)$ for every $t>0$. 
Now $\phi'_\infty(x_0)<\infty$ and $\phi(x_0, t) \le \phi'_\infty(x_0) t$ give 
$\phi^*(x_0, s) \ge \infty \chi_{(\phi'_\infty(x_0), \infty)}(s)$ and $\phi^*(x_0,w(x_0)) = \infty$. From this and 
$\rho_{\phi^*}(w)<\infty$ 
it follows that $w\le \phi'_\infty$ almost everywhere. However, 
$\phi'_\infty$ need not be continuous, so we cannot directly conclude that the inequality holds everywhere. 

Let $\omega$ be from \VA{} for $K:=1$. 
Choose $r_0>0$ and $\beta:= \frac1{1+ \omega(r_0)}$ such that $\phi'_\infty(x_0) < \beta^3 w(x_0) < \beta^2 w(x)$ 
for every $x\in B(x_0, r_0)$. 
Note that $\phi^*(x_0, \beta^3 w(x_0)) = \infty$ and $\phi(x_0,t)\le \phi'_\infty(x_0)t\le \beta^2 w(x) t$ 
for all $t\ge 0$. 
Since $\phi(x_0,\cdot)$ is finite and convex, it is continuous and we can find $t_x$ with 
$\phi(x_0,t_x)=|x-x_0|^{-n}$. By restricted \VA{},
\[
\phi(x, \beta t_x)\le \phi(x_0, t_x) + \omega(r_0)
=
\phi(x_0, t_x)+\tfrac1\beta -1 \le \beta^2 w(x) t_x +\tfrac1\beta. 
\]
By the definition of $\phi^*$ and the previous inequalities, we obtain that 
\[
\phi^*(x, w(x)) 
\ge 
\beta t_x w(x) - \phi(x, \beta t_x)
\ge
\beta(1-\beta) w(x)t_x  - \tfrac1\beta
\ge
\tfrac{1-\beta}\beta \phi(x_0, t_x) - \tfrac1\beta.
\]
Since $\phi(x_0, t_x) = |x-x_0|^{-n}$, we conclude that 
\[
\int_\Omega \phi^*(x, w)\, dy 
\gtrsim
\int_{B(x_0,r_0)} |x-x_0|^{-n}\, dy - c = \infty.
\]
This contradicts the assumption $\rho_{\phi^*}(w)<\infty$ and thus the counter-assumption 
$w(x_0)>\phi'_\infty(x_0)$ was incorrect and the claim is proved. 
\end{proof}

%\begin{rem}
%If in the previous lemma we have only the ordinary \aone{} condition, then 
%we conclude by the same argument that $\beta w \le \phi'_\infty$ with $\beta$ from \aone{}. 
%\end{rem}

%\begin{lem}\label{lem:lsc}
%Let $\phi\in \Phic(\Omega)$ satisfy \aone{} and restricted \VA{}. Then 
%$\phi'_\infty$ is lower semicontinuous.
%\end{lem}
%\begin{proof}
%Consider first $x\in \Omega$ with $\phi'_\infty(x)<\infty$. For $\epsilon>0$ and $\beta\in (0,1)$ 
%we can find $t>\frac1\epsilon$ such that $\frac{\phi(x, \beta t)}{\beta t} > \phi'_\infty(x)-\epsilon$. 
%For sufficiently small $|x-y|$, we have $\phi(y, t) \ge \phi(x, \beta t) - 1$ by restricted \VA{}. 
%Hence 
%\[
%\phi'_\infty(y) 
%\ge 
%\frac{\phi(y, t)}{t}
%\ge 
%\frac{\phi(x, \beta t) - 1}{t}
%\ge 
%\beta \frac{\phi(x, \beta t)}{\beta t} - \epsilon
%\ge 
%\beta( \phi'_\infty(x)-\epsilon) - \epsilon. 
%\]
%When $\epsilon\to 0^+$ and $\beta\to 1^-$, we obtain $\liminf_{y\to x}\phi'_\infty(y) \ge \phi'_\infty(x)$. 
%
%Let us then consider the case $\phi'_\infty(x)=\infty$. For $\beta$ from \aone{} and any $M> 0$, we can find 
%$t>0$ with $\frac{\phi(x, \beta t)}{\beta t} > M$. 
%It follows by \aone{} for sufficiently small 
%$|x-y|$ that $\phi(y, t) +1 \ge \phi(x, \beta t)$. 
%As before, $\phi'_\infty(y) \ge \beta M-1$. When $M\to \infty$, we obtain $\liminf_{y\to 
%x}\phi'_\infty(y)=\infty=\phi'_\infty(x)$. 
%\end{proof}

We define 
\[
T^\phi := \big\{ w \in C^1_0(\Omega; \Rn)\,\big|\, \rho_{\phi^*}(|w|)<\infty \big\}.
\]
Then the usual test function space of $\BV$ is $T^1$ since $\rho_{\infty}(|w|)<\infty$ if and 
only if $|w|\le 1$ a.e. In the next propositions we first consider the singular and absolutely continuous parts 
of the derivative separately. Then we combine them to handle the whole function in Theorem~\ref{thm:exactFormula}.

%\comment{
%This follows from the next proof, if it is even needed: 
%Let $\phi\in \Phic(\Omega)$ satisfy \aone{} and restricted \VA{}. Then 
%$\phi'_\infty$ is lower semicontinuous.
%}

\begin{prop}\label{prop:singularPart}
Let $\phi\in \Phic(\Omega)$ satisfy \azero{}, \adec{} and restricted \VA{}. 
If $u\in \BV(\Omega)$, then 
\[
\sup_{w\in T^\phi} \int_\Omega w\cdot dD^su = \int_\Omega \phi'_\infty \, d|D^su|.
\]
\end{prop}
\begin{proof}
By the definition of the total variation of a measure and Lemma~\ref{lem:bound}, 
\[
\sup_{w\in T^\phi} \int_\Omega w\cdot dD^su 
\le 
\sup_{w\in T^\phi} \int_\Omega |w|\,d|D^su| 
\le
\int_\Omega \phi'_\infty \, d|D^su|.
\]
For the opposite inequality, we define $h_k:\Omega\to[0,\infty]$ by
\[
h_k(x):=\lim_{r\to 0^+}\inf_{y\in B(x,r)} \frac{\phi(y, k)}k.
\]
Then $h_k$ is lower semicontinuous with $h_k\le \frac{\phi(\cdot, k)}k\le \phi'_\infty$. 
From the first inequality it follows that $\phi^*(\cdot, h_k) \le \phi(\cdot, k)$ so 
$\rho_{\phi^*}(h_k)\le \rho_\phi(k)<\infty$ since $\phi$ satisfies \azero{} and \adec{} and $\Omega$ is bounded. 
Let us show that $h_k\to \phi'_\infty$. If $\phi'_\infty(x)=\infty$, then since 
$\phi^+(k)<\infty$ we can use \aone{} in all sufficiently small balls to conclude that 
\[
h_k(x)=\lim_{r\to 0^+}\inf_{y\in B(x,r)} \frac{\phi(y, k)}k \ge \frac{\phi(x, \beta k) -1 }k \to \beta 
\phi'_\infty(x)=\infty
\]
as $k\to \infty$. If $\phi'_\infty(x)<\infty$, then we use the same inequality but now with $\beta:= \frac{1}{1+ 
\omega(r)}$ from 
the restricted \VA{} condition; we obtain the desired convergence as $\beta\to 1^-$. 

Note that $h_k$ is increasing in $k$ since $\phi$ is convex.
It follows by monotone convergence that 
\[
\int_\Omega \phi'_\infty \, d|D^su|
=
\lim_{k\to\infty} \int_\Omega h_k\,d|D^su|.
\]
Let $\epsilon>0$ and assume $\int_\Omega \phi'_\infty \, d|D^su|<\infty$. 
We can find $h=h_k$ and $K>0$ such that 
\[
\int_\Omega \phi'_\infty \, d|D^su|
\le
\int_\Omega h\,d|D^su| + \epsilon
\le
\sum_{j=1}^{K^2} \int_\Omega \tfrac1K \chi_{\{h>\frac jK\}}\,d|D^su| + 2\epsilon
=
\sum_{j=1}^{K^2} \tfrac1K |D^su|\big(\{h>\tfrac jK\}\big) + 2\epsilon.
\]
Since $h$ is lower semicontinuous, $\{h>\tfrac jK\}$ is open, and hence
%by Lemma~\ref{lem:lsc-approx}, there exists $w_j^i\in C^1_0(\{h>\tfrac jK\}; \Rn)$ 
%with $|w_j^i|\le \chi_{\{h>\frac jK\}}$ such that $|w_j^i| \to \chi_{\{h>\frac jK\}}$ as $i \to \infty$. 
%We may choose $(|w_j^i|)$ to be increasing, and hence monotone convergence theorem implies that
%\[
%\int_\Omega \chi_{\{h>\frac jK\}}\,d|D^su| = \lim_{i \to \infty}
%\int_\Omega |w_j^i|\,d|D^su|.
%\]
%For every $j$ we choose $i$, and denote the function by $w_j$, such that 
% \[
%\int_\Omega \chi_{\{h>\frac jK\}}\,d|D^su| \ge 
%\int_\Omega |w_j| \,d|D^su| + \frac{\epsilon}{K}.
%\]
%We obtain 
by Lemma~\ref{lem:singularPart} we can choose 
$w_j\in C^1_0(\{h>\tfrac jK\}; \Rn)$ with $|w_j|\le 1$ such that 
\[
\int_\Omega \phi'_\infty \, d|D^su|
\le
\sum_{j=1}^{K^2} \frac1K \int_{\{h>\tfrac jK\}} w_j\cdot dD^su + 3\epsilon
=
\int_\Omega \bigg(\underbrace{\sum_{j=1}^{K^2} \tfrac1K w_j}_{=:w_\epsilon}\bigg) \cdot dD^su + 3\epsilon.
\]
Note that $w_\epsilon \in C^1_0(\Omega; \Rn)$ and 
\[
|w_\epsilon| \le \sum_{j=1}^{K^2} \tfrac1K |w_j| \le \sum_{j=1}^{K^2} \tfrac1K \chi_{\{h>\frac jK\}} \le h
\]
so that $\rho_{\phi^*}(|w_\epsilon|)<\infty$. 
Therefore $w_\epsilon\in T^\phi$ and 
\[
\int_\Omega \phi'_\infty \, d|D^su|
\le
\int_\Omega w_\epsilon\cdot dD^su + 3\epsilon
\le
\sup_{w\in T^\phi} \int_\Omega w\cdot dD^su +3\epsilon.
\]
The upper bound follows from this as $\epsilon\to 0^+$. 
If $\int_\Omega \phi'_\infty \, d|D^su|=\infty$, then a similar argument gives 
$\frac1{3\epsilon}\le \sup_{w\in T^\phi} \int_\Omega w\cdot dD^su$ and the claim again follows.
\end{proof}

In the next result we assume that $\phi$ is continuous in both variables. Removing 
this somewhat unusual requirement is an open problem. Similar to the case of $V_\phi$ 
in Theorem~\ref{thm:BV-gradient}(2), the 
approximation is made much more difficult by the fact that $\phi^*$ is not doubling. 

\begin{prop}\label{prop:exactFormulaAC}
Let $\phi\in \Phic(\Omega) \cap C(\Omega\times [0,\infty))$ satisfy \azero{} and \adec{}.
If $u\in \BV(\Omega)$, then 
\[
\sup_{w\in T^\phi} \int_\Omega \nablaa u\cdot w - \phi^*(x, |w|)\, dx 
= 
\rho_\phi(|\nablaa u|).
\]
\end{prop}

\begin{proof}
The upper bound follows directly from Young's inequality:
\[
\sup_{w\in T^\phi} \int_\Omega \nablaa u\cdot w - \phi^*(x, |w|)\, dx 
\le 
\int_\Omega \phi(x, |\nablaa u|)\, dx.
\]
For the lower bound we make several reductions. Choose $g_i\in C(\Omega; \Rn)\cap L^\phi(\Omega; \Rn)$ 
with $g_i\to \nablaa u$  pointwise a.e and in $L^1(\Omega; \Rn)$. Then Fatou's Lemma and $L^1$-convergence yield
\[
\int_\Omega \phi(x, |\nablaa u|)\, dx\le \liminf_{i \to \infty} \int_\Omega \phi(x, |g_i|)\, dx
\quad\text{and}\quad
\lim_{i \to \infty} \int_\Omega g_i\cdot w \, dx = \int_\Omega \nablaa u\cdot w \, dx 
\]
when $w\in T^\phi$. Thus it suffices to show that 
\[
\int_\Omega \phi(x, |g|)\, dx
\le
\sup_{w\in T^\phi} \int_\Omega g\cdot w - \phi^*(x, |w|)\, dx 
\]
for $g\in C(\Omega; \Rn)\cap L^\phi(\Omega; \Rn)$. 
Furthermore, replacing $w$ by $\frac g{\epsilon+|g|} |w|$ and letting $\epsilon\to 0^+$, we see 
that $g\cdot \frac g{\epsilon+|g|} |w|\to |g| \, |w|$ 
pointwise. Thus by monotone convergence the vector-values of $g$ and $w$ are irrelevant and we 
need only show that 
\[
\int_\Omega \phi(x, |g|)\, dx
\le
\sup_{w\in C^1_0(\Omega)} \int_\Omega |g w| - \phi^*(x, |w|)\, dx 
\]
for $g\in C(\Omega)\cap L^\phi(\Omega)$. We also exclude test-functions with $\rho_{\phi^*}(w)=\infty$ by 
Remark~\ref{rem:restrictedTestFunction}.

Let $\phi'$ be the left-derivative of $\phi$ with respect to the second variable. 
Then $\phi'$ is left-continuous and $\phi(x,s)\ge \phi(x,s_0)+\phi'(x,s_0)(s-s_0)$ by convexity. 
We would like to choose $w:=\phi'(x, |g|)$ 
in the previous supremum. However, this function is not in general smooth and we cannot use regular 
approximation by smooth functions since $\phi^*$ is not doubling. Instead we define 
\[
\psi_\epsilon(x,t) := 
\int_{-\infty}^\infty \phi(x, \max\{\tau,0\}) \zeta_\epsilon(t-\tau)\, d\tau
= 
\phi*_t\zeta_\epsilon(x, t),
\]
where $\zeta_\epsilon$ is a mollifier in $\R$ with support in $[0, \epsilon]$. 
Since $\phi$ and $\phi'$ are increasing in the second variable and left-continuous, we see that 
$\psi_\epsilon\nearrow \phi$ and $\psi_\epsilon'\nearrow \phi'$ as $\epsilon\to 0^+$. 
Furthermore, 
$\psi_\epsilon' = \phi*_t(\zeta_\epsilon')$ is continuous in $x$ since $\phi$ is and it continuous 
in $t$ as a convolution with a smooth function. 
Let $v_i \in C_0(\Omega)$ with $0\le v_i \le 1$ and $v_i\nearrow 1$. 
By uniform continuity in $\supp v_i$, we can choose $\delta=\delta_{\epsilon, i}>0$ such that 
$\psi_\epsilon'(x, |g(x)|\,v_i(x))-\epsilon\le \psi_\epsilon'(y, |g(y)|\,v_i(y))$ for all $x\in B(y,\delta)$ and $y\in\Omega$. 
Then
\[
w_{\epsilon, i}:= \max\{\psi_\epsilon'(\cdot, |g|\, v_i)-\epsilon, 0\} *_x \eta_\delta
\le \psi_\epsilon'(\cdot, |g|)\le \phi'(\cdot, |g|).
\]

Now $w_{\epsilon,i}\to \phi'(\cdot, |g|)$, so we conclude by Fatou's Lemma that
\[
\int_\Omega |g| \phi'(x, |g|) \, dx \le \liminf_{i\to\infty, \epsilon\to 0}\int_\Omega |g w_{\epsilon, i}| \, dx.
\]
Since $\phi$ satisfies \azero{} and \adec{}, we see that
\[
\phi^*(x, |w_{\epsilon,i}|) 
\le 
\phi^*(x, \phi'(x, |g|))
\le 
|g| \phi'(x, |g|)
\lesssim 
\phi(x, |g|).
\]
As $g\in L^\phi(\Omega)$ and $\phi$ satisfies \adec{}, $\rho_\phi(g)<\infty$. 
Thus dominated convergence with majorant $c \phi(\cdot, g)$ yields
\[
\int_\Omega \phi^*(x, \phi'(x, |g|))\, dx = 
\lim_{i\to\infty, \epsilon\to 0}\int_\Omega \phi^*(x, |w_{\epsilon,i}|)\, dx.
\]
Since $w_\epsilon$ is a valid test-function and $\phi'(\cdot, |g|)<\infty$ a.e., 
this together with ``Young's equality'' (Lemma~\ref{lem:dual-equality}) implies that 
\begin{align*}
\sup_{w\in C^1_0(\Omega)} \int_\Omega |g w| - \phi^*(x, |w|)\, dx 
&\ge
\liminf_{i\to\infty, \epsilon\to 0}\int_\Omega |g| \, |w_{\epsilon,i}| - \phi^*(x, |w_{\epsilon,i}|)\, dx \\
&\ge
\int_\Omega |g| \phi'(x, |g|) - \phi^*(x, \phi'(x, |g|))\, dx 
=
\int_\Omega \phi(x, |g|) \, dx. 
\end{align*}
This completes the proof of the lower bound.
\end{proof}

We next derive a simple, closed form expression for $\rho_{V,\phi}$. This is the main result of the paper. 
Note that the right-hand side expression was also obtained recently in the one-dimensional case 
for a modular based on the Riesz variation assuming the \VA{} condition \cite{HasJR_pp}.

\begin{thm}\label{thm:exactFormula}
Let $\phi\in \Phic(\Omega) \cap C(\Omega\times [0,\infty))$ satisfy \azero{}, \adec{} and restricted \VA{}. 
If $u\in \BV(\Omega)$, then 
\[
\rho_{V,\phi}(u) = \rho_\phi(|\nablaa u|) + \int_\Omega \phi'_\infty \, d|D^su|.
\]
\end{thm}
\begin{proof}
Since $Du= D^a u + D^su$, integration by parts implies that 
\[
- \int_\Omega u \div w\, dx 
=
\int_\Omega w\cdot dDu
= 
\int_\Omega \nablaa u\cdot w\, dx + \int_\Omega w \cdot dD^su
\]
for $w\in T^\phi$. 
Hence the claim follows from Propositions~\ref{prop:singularPart} and \ref{prop:exactFormulaAC}
once we prove that
\begin{align*}
&\sup_{w\in T^\phi} \bigg[\int_\Omega \nablaa u\cdot w - \phi^*(x, |w|)\, dx 
+
\int_\Omega w \cdot dD^su\bigg]\\
&\qquad=
\sup_{w\in T^\phi} \int_\Omega \nablaa u\cdot w - \phi^*(x, |w|)\, dx 
+
\sup_{w\in T^\phi} \int_\Omega w \cdot dD^su. 
\end{align*}
The inequality ``$\le$'' is clear, so we focus on the opposite one. 

We assume first that $\rho_\phi(|\nablaa u|) + \int_\Omega \phi'_\infty \, d|D^su|<\infty$ and fix $\epsilon>0$. 
By the definition of supremum we can choose $w_1, w_2\in T^\phi$ such that 
\[
\sup_{w\in T^\phi} \int_\Omega \nablaa u\cdot w - \phi^*(x, |w|)\, dx
\le 
\int_\Omega \nablaa u\cdot w_1 - \phi^*(x, |w_1|)\, dx + \epsilon < \infty
\]
and
\[
\sup_{w\in T^\phi} \int_\Omega w \cdot dD^su
\le 
\int_\Omega w_2 \cdot dD^su + \epsilon< \infty.
\]
Since $u \in \BV(\Omega)$ and $w_i\in T^\phi$, 
we have $|\nablaa u |\,| w_i| \in L^1(\Omega)$ and $\rho_{\phi^*}(|w_i|)<\infty$.
Thus, by the absolute continuity of the integral, we find $\delta>0$ such that 
\[
\bigg|\int_{\Omega\setminus \Omega_1} \nablaa u \cdot w_i-\phi^*(x, |w_i|)\, dx \bigg|
\le 
\int_{\Omega\setminus \Omega_1} |\nablaa u|\,|w_i|+\phi^*(x, |w_i|)\, dx 
\le 
\epsilon
%\quad\text{and}\quad
%\int_{\Omega\setminus \Omega_1} \phi(x,|\nablaa u|)\, dx 
%\le 
%\epsilon
\]
for $i\in \{1,2\}$ and any $\Omega_1\subset \Omega$ with $|\Omega\setminus \Omega_1|<\delta$ and 
\[
\bigg|\int_{\Omega\setminus \Omega_2} w_2\cdot dD^su \bigg| 
\le 
\int_{\Omega\setminus \Omega_2} \phi'_\infty\, d|D^su|
\le 
\epsilon
\]
for any $\Omega_2\subset \Omega$ with $|D^su|(\Omega\setminus \Omega_2)<\delta$.

Since $\supp D^su$ has Lebesgue measure zero, we can find a finite collection of open rectangles 
$Q_i\subset\Omega$ with 
$|D^su|(\bigcup Q_i) > |D^su|(\Omega)-\delta$ and $|\bigcup 2Q_i|<\delta$. Then we choose $\theta\in C^1_0(\Omega)$ 
with $0\le \theta\le 1$, $\theta=1$ in $\Omega_2:=\bigcup Q_i$ 
and $\theta=0$ in $\Omega_1:=\Omega\setminus\bigcup 2Q_i$. 
Let $w_\epsilon:=\theta w_2 + (1-\theta)w_1\in C^1_0(\Omega; \Rn)$. 
Since $w_\epsilon$ is a pointwise convex combination, 
%$0\le \theta\le 1$ we have $|w_\epsilon(x)| \le \theta(x)|w_2(x)| + (1-\theta(x))|w_1(x)| \le \max\{|w_2(x)|, |w_1(x)|\}$, and thus
\[
 \phi^*(\cdot, |w_\epsilon|) \le \phi^*(\cdot, \max\{|w_2|, |w_1|\})
\le  \phi^*(\cdot, |w_2|) + \phi^*(\cdot, |w_1|).
\] 
This yields that   $\rho_{\phi^*}(|w_\epsilon|) \le \rho_{\phi^*}(|w_2|) + \rho_{\phi^*}(|w_1|)< \infty$, 
and so $w_\epsilon \in T^\phi$. By Lemma~\ref{lem:bound}, $|w_\epsilon|\le \phi'_\infty$. We obtain that
\begin{align*}
\rho_{V,\phi}(u) 
&\ge
\int_\Omega \nablaa u\cdot w_\epsilon - \phi^*(x, |w_\epsilon|)\, dx + \int_\Omega w_\epsilon \cdot dD^su \\
& \ge 
\int_{\Omega_1} \nablaa u\cdot w_1 - \phi^*(x, |w_1|)\, dx 
+ \int_{\Omega_2} w_2 \cdot dD^su - c_\theta \\
& \ge 
\int_{\Omega} \nablaa u\cdot w_1 - \phi^*(x, |w_1|)\, dx 
+ \int_{\Omega} w_2 \cdot dD^su - 5\epsilon,
\end{align*}
where 
\begin{align*}
c_\theta
&:=
\int_{\Omega\setminus\Omega_1} |\nablaa u|\,|w_\epsilon| + \phi^*(x, |w_\epsilon|)\, dx 
+ \int_{\Omega\setminus\Omega_2} |w_\epsilon|\,  d|D^su| 
%&\le 
%\int_{\Omega\setminus\Omega_1} \phi(x,|\nablaa u|) \, dx 
%+ \int_{\Omega\setminus \Omega_2} \phi'_\infty \, d|D^su| 
\le 
3\epsilon
\end{align*}
by the absolute integrability assumptions. 
By the choice of $w_1$ and $w_2$,  
\[
\rho_{V,\phi}(u) 
\ge 
\sup_{w\in T^\phi} \int_\Omega \nablaa u\cdot w - \phi^*(x, |w|)\, dx 
+
\sup_{w\in T^\phi} \int_\Omega w \cdot dD^su - 7\epsilon. 
\]
The lower bound follows as $\epsilon\to 0^+$. This concludes the 
proof in the case $\rho_\phi(|\nablaa u|) + \int_\Omega \phi'_\infty \, d|D^su|<\infty$.

When $\rho_\phi(|\nablaa u|)=\infty$ and $\int_\Omega \phi'_\infty \, d|D^su|<\infty$, 
we estimate 
\begin{align*}
&\sup_{w\in T^\phi} \bigg[\int_\Omega \nablaa u\cdot w - \phi^*(x, |w|)\, dx 
+
\int_\Omega w \cdot dD^su\bigg]\\
&\qquad\ge
\sup_{w\in T^\phi} \int_\Omega \nablaa u\cdot w - \phi^*(x, |w|)\, dx 
-
\sup_{w\in T^\phi} \int_\Omega w \cdot dD^su \\
&\qquad=
\rho_\phi(|\nablaa u|)-\int_\Omega \phi'_\infty \, d|D^su| = \infty. 
\end{align*}

Only the case $\int_\Omega \phi'_\infty \, d|D^su|=\infty$ remains. By the proof of 
Proposition~\ref{prop:singularPart}, there exists $w_\epsilon\in C^1_0(\Omega; \Rn)$ with 
\[
\int_\Omega w_\epsilon \cdot dD^su > \frac1\epsilon
\]
and $|w_\epsilon|\le \frac{\phi(\cdot, k)}k$ for some $k=k_\epsilon$. 
For any $\theta:\Omega\to [0,1]$, we find that 
\[
\begin{split}
\nablaa u \cdot (\theta w_\epsilon) - \phi^*(\cdot, |\theta w_\epsilon|)
&\ge \nablaa u \cdot (\theta w_\epsilon) - \phi^* \Big(\cdot, \frac{\phi(\cdot, k)}k\Big)
\ge -\Big(\frac{\phi^+(k)}k \,|\nablaa u| + \phi^+(k) \Big).
\end{split}
\] 
Since the function on the right-hand side is integrable, we can choose $\delta_k>0$ such that 
its integral over any measurable $A$ with $|A|<\delta_k$ is at least $-1$. 
Furthermore, since $\supp D^su$ has measure zero, we can choose $\theta\in C^\infty_0(\Omega)$ as 
before to have support with Lebesgue measure at most $\delta_k$ and satisfy 
\[
\int_\Omega \theta w_\epsilon \cdot dD^su > 
\frac1{2}\int_\Omega w_\epsilon \cdot dD^su > \frac1{2\epsilon}. 
\] 
Then 
\begin{align*}
&\sup_{w\in T^\phi} \bigg[\int_\Omega \nablaa u\cdot w - \phi^*(x, |w|)\, dx 
+
\int_\Omega w \cdot dD^su\bigg]\\
&\qquad \ge
\int_\Omega \theta w_\epsilon \cdot dD^su 
+ \int_\Omega \nablaa u\cdot (\theta w_\epsilon) - \phi^*(x, |\theta w_\epsilon|)\, dx 
\ge
\frac1{2\epsilon} - 1.
\end{align*}
When $\epsilon \to\infty$, the claim follows in this case also. 
\end{proof}

As a special case we obtain the following result in Orlicz spaces. Now the recession 
function is just a constant, either finite or infinite. 
As can be seen, we do not obtain any new spaces in this case, only the classical 
$\BV$-space or the regular Sobolev space. 

\begin{cor}\label{cor:Orlicz}
Let $\phi\in \Phic$ be independent of $x$ and satisfy \adec{}. 
If $u\in \BV(\Omega)$, then 
$\rho_{V,\phi}(u) = \rho_\phi(|\nablaa u|) + \phi'_\infty\, |D^su|(\Omega)$ and so 
\begin{enumerate}
\item
$\BV^\phi(\Omega)=\BV(\Omega)$ if $\phi'_\infty < \infty$;
\item
$\BV^\phi(\Omega)=W^{1,\phi}(\Omega)$ if $\phi'_\infty = \infty$.
\end{enumerate}
\end{cor}

%%%%%%%%%%%%%%%%%%%%%%%%%%%%%%%%%%%%%%%%%%%%%%%%%%%%%%%%
%%%%%%%%%%%%%%%%%%%%%%%%%%%%%%%%%%%%%%%%%%%%%%%%%%%%%%%%
%%%%%%%%%%%%%%%%%%%%%%%%%%%%%%%%%%%%%%%%%%%%%%%%%%%%%%%%

\section{Precise approximation and \texorpdfstring{$\Gamma$}{Gamma}-convergence}
\label{sect:Gamma}

We can now prove a precise approximation lemma for the modular using the formula 
for $\rho_{V, \phi}$ from the previous section. In contrast to 
Lemma~\ref{lem:density} which provides only approximate equality of the limit we 
here obtain that the limit exactly equals $\rho_{V,\phi}(u)$, under appropriately stronger 
assumptions  on $\phi$. This is critical for $\Gamma$-convergence. 
A similar argument should also work for $V_\phi$ with the same assumptions. 

Note that we assume \VA{} for $\phi^*$, not only its restricted version. 
This is used for Young's convolution inequality. In \cite[Lemma~4.1.7]{HarH19} it was shown that 
\aone{} of $\phi$ and $\phi^*$ are equivalent provided \azero{} holds; the corresponding 
statement is not known for \VA{}.

\begin{prop}[Modular approximation by smooth functions]\label{prop:modularDensity}
Let $\phi\in \Phic(\Omega) \cap C(\Omega\times [0,\infty))$ satisfy \azero{}, \adec{} and restricted \VA{} 
and assume that $\phi^*$ satisfies \VA{}. 
For every $u \in L^\phi(\Omega)$ there exist $u_k \in C^\infty(\Omega)$ such that 
\[
u_k \to u \text{ in }L^\phi(\Omega)
\quad\text{and}\quad
\rho_{V,\phi}(u) = \lim_{k \to \infty} \rho_\phi(|\nabla u_k|).
\]
If additionally $u \in L^2(\Omega)$, then the sequence can be chosen with
$u_k \to u$ in $L^2(\Omega)$ as well.
\end{prop}
\begin{proof}
Since the case $\rho_{V,\phi}(u)=\infty$ is trivial, we may assume that $\rho_{V,\phi}(u)<\infty$.
Let $\epsilon \in (0,1)$. We define $\xi_k$, $\eta_{\epsilon_k}$, and $u_\epsilon$ as 
in the proof of Lemma~\ref{lem:density} so that $V_\phi(u_\epsilon)\lesssim V_\phi(u)$. It follows from \adec{q} that 
\[
\min\{ \|u\|_{\rho_{V,\phi}}, \|u\|_{\rho_{V,\phi}}^q\}
\lesssim
\rho_{V,\phi}(u) 
\lesssim
\max\{ \|u\|_{\rho_{V,\phi}}, \|u\|_{\rho_{V,\phi}}^q\}.
\]
Thus Lemma~\ref{lem:equivalence} and $V_\phi(u_\epsilon)\lesssim V_\phi(u)$ imply that 
$\rho_\phi(|\nabla u_\epsilon|)\lesssim \rho_{V,\phi}(u)^q+1$. 
From Theorem~\ref{thm:exactFormula} we see that $U_1$ can be chosen so large (by 
choosing $m$ large in Lemma~\ref{lem:density}) that $\rho_{V\setminus\overline{U_1},\phi}(u)<\epsilon$. 
Then $V_\phi(u, \Omega\setminus \overline{U_1})\lesssim \epsilon^{1/q}$.

Since $u_\epsilon\in C^\infty(\Omega)$, $\nablaa u_\epsilon=\nabla u_\epsilon$. 
By the proof of Proposition~\ref{prop:exactFormulaAC} with $|\nabla u_\epsilon|$ as $g$, there exists 
$w_\epsilon\in C^1_0(\Omega; \Rn)$ with 
$\phi^*(x, |w_\epsilon|)\lesssim \phi(x,|\nabla u_\epsilon|)$ and
\[
\int_\Omega \phi(x,|\nabla u_\epsilon|)\, dx 
\le
(1+\epsilon)\int_\Omega \nabla u_\epsilon\cdot w_\epsilon - \phi^*(x,|w_\epsilon|)\, dx.
\]
By \adec{} of $\phi$, Theorem~\ref{thm:exactFormula} and the estimates above, 
\[
\rho_{\phi^*}(|w_\epsilon|) \lesssim \rho_\phi(|\nabla u_\epsilon|) 
%\le \rho_{V,\phi}(u_\epsilon).
\lesssim \rho_{V,\phi}(u)^q+1.
\]
Thus $\|w_\epsilon\|_{\phi^*}\le c$. 
As in Lemma~\ref{lem:density}, we have 
\[
\begin{split}
\int_\Omega \nabla u_\epsilon \cdot w_\epsilon \, dx 
= \underbrace{\sum_{k=1}^\infty \int_\Omega u \div (\xi_k(\eta_{\epsilon_k}* w_\epsilon)) \, dx}_{=: I} - 
\underbrace{\sum_{k=1}^\infty \int_\Omega w_\epsilon \cdot (\eta_{\epsilon_k}*(u \nabla \xi_k) - (u\nabla \xi_k))\, 
dx}_{=: II}
\end{split}
\]
and the inequality $|II|\le c\epsilon$ again follows.

We divide the term $I$ into two parts.  Let $\omega$ be from Corollary~\ref{cor:convolution}. 
Using the definition of $\rho_{V,\phi}$ to the first part of $I$, and and estimating the second part of $I$ as in Lemma~\ref{lem:density}  
but now with a test-function supported in $\Omega\setminus \overline{U_1}$, we find that 
\[
\begin{split}
|I| & 
= 
\bigg| \int_\Omega u \div (\xi_1(\eta_{\epsilon_1}*w_\epsilon)) \, dx + 
\int_\Omega u \div \bigg( \sum_{k=2}^\infty  \xi_k(\eta_{\epsilon_k}*w_\epsilon)\bigg) \, dx \bigg|\\
&\le 
\rho_{V,\phi}\big((1+\omega(\epsilon_1)) u\big) + \rho_{\phi^*}\Big( \frac{\xi_1|\eta_{\epsilon_1}*w_\epsilon|}{1+\omega(\epsilon_1)}\Big) + 
cV_\phi(u, \Omega\setminus \overline{U_1}) \\
&\le 
\rho_{V,\phi}\big((1+\omega(\epsilon_1)) u\big) + \rho_{\phi^*}\Big( \frac{\eta_{\epsilon_1}*|w_\epsilon|}{1+\omega(\epsilon_1)}\Big) + 
cV_\phi(u, \Omega\setminus \overline{U_1})
\end{split}
\]
By Young's convolution inequality (Corollary~\ref{cor:convolution}), 
\[
\rho_{\phi^*}\Big( \frac{\eta_{\epsilon_1}*|w_\epsilon|}{1+\omega(\epsilon_1)}\Big) 
- \rho_{\phi^*}(|w_\epsilon|) 
\le \rho_{\phi^*}( |w_\epsilon|) + \omega(\epsilon_1) -\rho_{\phi^*}(|w_\epsilon|) 
\le \omega(\epsilon_1) 
\to 0
\]
as $\epsilon_1\to 0^+$. 
Combining the estimates, we obtain that 
\begin{align*}
\int_\Omega \phi(x,|\nabla u_\epsilon|)\, dx 
&\le
(1+\epsilon)\int_\Omega \nabla u_\epsilon\cdot w_\epsilon - \phi^*(x,|w_\epsilon|)\, dx\\
&\le
(1+\epsilon)(|I|-\rho_{\phi^*}(|w_\epsilon|)) + c\epsilon \\
&\le
(1+\epsilon)\rho_{V,\phi}\big((1+\omega(\epsilon_1)) u\big) + 
c(|\Omega|\,\omega(\epsilon_1)  + \epsilon^{1/q}). 
\end{align*}
By \cite[Lemma~2.2.6]{HarH19}, there exists a constant $q_2$ depending on $q$ such that 
$\rho_{V,\phi}\big((1+\omega(\epsilon_1)) u\big)\le (1+\omega(\epsilon_1))^{q_2}\rho_{V,\phi}(u)$.
As $\epsilon,\epsilon_1\to 0^+$, 
we obtain that $\limsup_{\epsilon\to 0^+}\rho_\phi(|\nabla u_\epsilon|)\le \rho_{V,\phi}(u)$.
The opposite inequality follows from Lemma~\ref{lem:sequence-in-BV} as in Lemma~\ref{lem:density}.
\end{proof}

In \cite{EleHH_pp}, we introduced abstract $\BV^\phi$-type spaces by a limit procedure. 
%We use the $^\circ$-symbol to indicate only the gradient part 
%of a Sobolev-type norm and the $\bar{\phantom{m}}$-symbol to indicate the $\liminf$-approximation process. 
%
%\begin{defn}\label{defn:norms}
%For $\phi\in\Phiw(\Omega)$ and $u \in L^1(\Omega)$ we define
%\[
%\|u\|\liminftag := 
%\inf\big\{\liminf_{i\to \infty}\|\nabla u_k\|_\phi \mid 
%u_k \in \Lspace{\Omega}\text{ and } u_k \to u \text{ in } L^1(\Omega) \big\},
%\]
%\[
%\rholiminf (u) := \inf\big\{\liminf_{i\to \infty}\rho_\phi(|\nabla u_k|) \mid 
%u_k \in \Lspace {\Omega}\text{ and }u_k \to u \text{ in } L^1(\Omega) \big\}.
%\]
%\[
%\rhoBV(u) := \rho_{1}(u) + \rholiminf(u) 
%\qquad\text{and}\qquad
%\Lbar{\Omega} := \{ u\in L^1(\Omega) : \rhoBV(\lambda u) <\infty \text{ for some } \lambda >0\}. 
%\]
%\end{defn}
%
We use here the version with a fidelity term which is most relevant for image processing. 
For $p >1$ and for a given $f \in L^2(\Omega)$, we defined functionals $F_p: L^2(\Omega) \to [0, 
\infty]$ by
\begin{equation*}
%\label{eq:Fp}
F_p(u) := 
\begin{cases}
\displaystyle\int_\Omega \phi(x, |\nabla u|)^p + |u-f|^2 \, dx & \text{when } u \in L^{1, \phi^p}(\Omega); \\
\infty & \text{otherwise.}
\end{cases}
\end{equation*}
and the limit functional $F: L^2 (\Omega) \to [0, \infty]$ by
\begin{equation*} %\label{defn:E}
F(u) := 
\inf\bigg\{\liminf_{i\to\infty}\int_\Omega \phi(x, |\nabla u_k|) + |u_k-f|^2\, dx \,\Big|\, 
u_k \in L^{1, \phi} (\Omega) \cap L^2(\Omega)\text{ and } u_k\to u \text{ in } L^2(\Omega) \bigg\}
\end{equation*}
%Note that $F$ is a variant of $\rholiminf$, this time including a ``fidelity term'' in $L^2$. 
Note that the energy in $F_p$ satisfies \ainc{} and \adec{} so it can be studied in a reflexive space and 
it is easy to prove existence of minimizers among other things \cite{HarHK16}. 

We compare $F$ with the corresponding version of $\rho_{V, \varphi}$ including the 
fidelity term, namely 
\[
\rho_{V,\phi}^f(u):=
\rho_{V,\phi}(u)+\rho_2(u-f) 
= 
\sup\bigg\{ \int_\Omega u \div w - \phi^*(x, |w|)+|u-f|^2\, dx \,\Big|\, w \in C^1_0(\Omega; \Rn)\bigg\}.
\]

\begin{prop}
\label{comparison-modular}
Let $\phi\in \Phic(\Omega) \cap C(\Omega\times [0,\infty))$ satisfy \azero{}, \adec{} and restricted \VA{} 
and assume that $\phi^*$ satisfies \VA{}. 
Then $\rho_{V,\phi}^f(u) \le F(u)$ for all $u \in L^2(\Omega)$ and $\rho_{V,\phi}^f(u) = F(u)$
for all $u \in L^2(\Omega)\cap L^\phi(\Omega)$.
\end{prop}

\begin{proof}
Let us prove first that $\rho_{V,\phi}^f(u) \le F(u)$. 
We may assume $F(u) < \infty$ and consider functions $u_k \in L^{1, \phi}(\Omega) 
\cap L^2(\Omega)$ realizing the infimum from $F$ with $u_k \rightarrow u$ in $L^2(\Omega)$. 
Weak lower semicontinuity of $\rho_{V,\phi}$ (Lemma~\ref{lem:sequence-in-BV}) and in $L^2$ gives
\[
\rho_{V,\phi}^f(u) \le \liminf_{i\to \infty} \rho_{V,\phi}^f(u_k).
\]
By Young's inequality, $\rho_{V,\phi}^f(u_k) \le \rho_{\phi}(|\nabla u_k|) + \rho_2(u_k - f)$
so that
\[
\rho_{V,\phi}^f(u) \le \liminf_{i\to \infty} \big(\rho_{\phi}(|\nabla u_k|) + \rho_2(u_k - f) \big) = F(u).
\]
Thus the inequality is proved.

For the opposite inequality, $F(u) \le \,\rho_{V,\phi}^f(u)$, 
we may assume that $\rho_{V,\phi}^f(u) < \infty$.
By Proposition~\ref{prop:modularDensity}, there exist $u_k \in C^\infty(\Omega)$ such that 
$u_k \to u$ in $L^{\varphi}(\Omega)\cap L^2(\Omega)$ and
\[ 
\rho_{V,\phi}(u) = \lim_{i \to \infty} \rho_\phi(|\nabla u_k|).
\]
Since $\rho_\phi(|\nabla u_k|)<\infty$ and $u_k\in L^1(\Omega)$, we see that $u_k \in L^{1, \phi}(\Omega)$,
and so, by the definition of $F$, using the fact that the limit of the sum is the sum of the limits, 
we obtain that
\[
F(u)  \le  \liminf_{i \to \infty} \big(\rho_\phi(|\nabla u_k|)  + \rho_2(u_k - f) \big)
= \rho_{V,\phi}^f(u). \qedhere
\]
\end{proof}

The concept of $\Gamma$-convergence was introduced by De Giorgi and Franzoni \cite{DeGF75}, see also 
\cite{Bra02, Dal93}.
A family of functionals $F_p: L^2(\Omega) \to [0, \infty]$ is said to \textit{$\Gamma$-converge}
to $F: L^2(\Omega) \to [0, \infty]$ in $L^2(\Omega)$
if the following hold for every sequence $(p_k)$ converging to one from above:
\begin{enumerate}
 \item[(a)]  $\displaystyle F (u) \le \liminf_{i \to \infty} F_{p_k} (u_{i})$ 
for every $u \in L^2(\Omega)$ and every sequence with $u_{i}\to u$ in $L^2(\Omega)$;
 \item[(b)]
$\displaystyle F (u) \ge \limsup_{i \to \infty} F_{p_k} (u_{i})$
for every $u \in L^2(\Omega)$ and some sequence with $u_{i}\to u$ in $L^2(\Omega)$.
\end{enumerate}
%The sequence in (b) is called the \textit{recovery sequence}.

We conclude by showing the $\Gamma$-convergence in the 
situation most relevant to image processing: convex planar domains. 
This allows us to simplify the assumptions.

\begin{cor}
Let $\Omega\subset\R^2$ be convex and let $\phi\in \Phic(\Omega)$ 
satisfy \azero{}, \adec{2} and \VA{} and assume that $\phi^*$ satisfies \VA{}. 
Then $F_p$ $\Gamma$-converges to $\rho_{V,\phi}^f$ in $L^2(\Omega)$.
\end{cor}

\begin{proof}
To establish the necessary conditions we use some results from references without defining here all 
the terms. The references can be consulted if necessary. By \cite[Corollary~4.6]{Juu_pp}, 
$C^\infty(\overline{\Omega})$ is dense in $W^{1, \phi}(\Omega)$ if $\Omega$ is an 
$(\epsilon, \delta)$-domain and $\phi$ satisfies \azero{}, \aone{} and (A2). 
We note that (A2) holds since $\Omega$ is bounded 
\cite[Lemma~4.2.3]{HarH19} and $\Omega$ is an $(\epsilon, \delta)$-domain since it is convex.

Since $\phi$ satisfies \adec{2}, $L^2(\Omega) \subset L^\phi(\Omega)$ and thus
$L^{1, \phi}(\Omega) \cap L^2(\Omega) \hookrightarrow W^{1, \phi}(\Omega)$. Since the dimension is $2$ we also have  
$W^{1, \phi}(\Omega) \hookrightarrow W^{1, 1} (\Omega) \hookrightarrow L^2(\Omega)$. 
Thus $C^\infty(\overline{\Omega})$ is dense in $L^{1, \phi}(\Omega) \cap L^2(\Omega)$ with respect to 
the norm $u \mapsto \|u\|_2 + \|\nabla u\|_\phi$. 
By density, 
\cite[Theorem~1.3(2)]{EleHH_pp} yields that $F_p$ $\Gamma$-converges to $F$ in $L^2(\Omega)$.
Since $\phi$ satisfies \VA{}, it belongs to $C(\Omega\times [0,\infty))$. 
Thus Proposition~\ref{comparison-modular} gives $F= \rho_{V,\phi}^f$ in $L^2(\Omega) = L^2(\Omega)\cap L^\phi(\Omega)$.
\end{proof}

%%%%%%%%%%%%%%%%%%%%%%%%%%%%%%%%%%%%%%%%%%%%%%%%%%%%%%%%%%%%%%%%%%%%%%%%%
%%%%%%%%%%%%%%%%%%%%%%%%%%%%%%%%%%%%%%%%%%%%%%%%%%%%%%%%%%%%%%%%%%%%%%%%%
%%%%%%%%%%%%%%%%%%%%%%%%%%%%%%%%%%%%%%%%%%%%%%%%%%%%%%%%%%%%%%%%%%%%%%%%%

\end{document}